\documentclass[12pt,reqno]{amsart}
\usepackage{fullpage}
\usepackage{amsmath}
\usepackage{mathrsfs,amssymb,graphicx,verbatim,hyperref}
\usepackage{paralist}

\renewcommand{\setminus}{\smallsetminus}

\usepackage{ifthen}

\newcommand{\ifsodaelse}[2]{\ifthenelse{\isundefined{\SODAF}}{#2}{#1}}

\ifsodaelse    {\usepackage{ltexpprt}\usepackage{amsmath}}{}
\usepackage{mathrsfs,amssymb,graphicx,verbatim}
\usepackage{paralist}
\newcommand{\K}{\mathcal K}
\newcommand\remove[1]{}
\renewcommand{\H}{\mathcal H}
\newcommand{\sign}{\mathrm{sign}}

\newcommand{\rnote}[1]{}
\newcommand{\jnote}[1]{}
\renewcommand{\S}{\mathbb S}

\newcommand{\e}{\varepsilon}
\newcommand{\R}{\mathbb{R}}

\newcommand{\E}{\mathbb{E}}

\newcommand{\N}{\mathbb{N}}

\newcommand{\Opt}{\mathrm{OPT}}
\newcommand{\SDP}{\mathrm{SDP}}

\newcommand{\C}{\mathbb{C}}

\newcommand{\f}{\varphi}

\newtheorem{theorem}{Theorem}[section]
\newtheorem{lemma}[theorem]{Lemma}

\newtheorem{corollary}[theorem]{Corollary}

\newtheorem{definition}[theorem]{Definition}

\newtheorem{remark}{Remark}[section]

\newtheorem{conjecture}[theorem]{Conjecture}
\newtheorem{question}[theorem]{Question}
\newcommand{\eqdef}{\stackrel{\mathrm{def}}{=}}
\date{}
\newcommand{\D}{\mathbb D}
\renewcommand{\le}{\leqslant}
\renewcommand{\ge}{\geqslant}

  \newtheorem{proposition}[subsection]{Proposition}

\include{psfig}

\begin{document}

\title{The Grothendieck constant is strictly smaller than Krivine's bound}
\thanks{M.~B. was supported in part by an NSERC Discovery Grant. A.~N. was supported in part by NSF grant CCF-0832795,
BSF grant 2006009, and the Packard Foundation. Part of this work was completed when A.~N. was visiting the Discrete Analysis program at the Isaac Newton Institute for Mathematical Sciences. An extended abstract describing the contents of this work will appear in the 52nd Annual IEEE Symposium on Foundations of Computer Science.}
\author{Mark Braverman}
\address{University of Toronto, 10 King's College Road, Toronto, ON M5S 3G4.}
\email{mbraverm@cs.toronto.edu}
\author{Konstantin Makarychev}
\address{IBM T.J. Watson Research Center, P.O. Box 218, Yorktown Heights, NY 10598.}
\email{konstantin@us.ibm.com}
\author{Yury Makarychev}
\address{Toyota Technological Institute at Chicago, 6045 S. Kenwood Ave., Chicago, IL 60637.}
\email{yury@ttic.edu }
\author{Assaf Naor}
\address{Courant Institute, 251 Mercer Street, New York, NY 10012.}
\email{naor@cims.nyu.edu}

\date{}
\maketitle

\begin{abstract}  The (real) Grothendieck constant $K_G$ is the infimum over those $K\in (0,\infty)$ such that for every $m,n\in \N$ and every $m\times n$ real matrix $(a_{ij})$ we have
$$
\max_{\{x_i\}_{i=1}^m,\{y_j\}_{j=1}^n\subseteq S^{n+m-1}}\sum_{i=1}^m\sum_{j=1}^na_{ij}\langle x_i,y_j\rangle
\le K \max_{\{\e_i\}_{i=1}^m,\{\delta_j\}_{j=1}^n\subseteq \{-1,1\}}\sum_{i=1}^m\sum_{j=1}^n a_{ij}\e_i\delta_j.
$$
The classical Grothendieck inequality asserts the non-obvious fact that the above inequality does hold true for some $K\in (0,\infty)$ that is independent of $m,n$ and $(a_{ij})$. Since Grothendieck's 1953 discovery of this powerful theorem, it has found numerous applications in a variety of areas, but despite attracting a lot of attention, the exact value of the Grothendieck constant $K_G$ remains a mystery. The last progress on this problem was in 1977, when Krivine proved that $K_G\le \frac{\pi}{2\log\left(1+\sqrt{2}\right)}$ and conjectured that his bound is optimal. Krivine's conjecture has been restated repeatedly since 1977, resulting in focusing the subsequent research on the search for examples of matrices $(a_{ij})$ which exhibit (asymptotically, as $m,n\to \infty$) a lower bound on $K_G$ that matches Krivine's bound. Here we obtain an improved Grothendieck inequality that holds for {\em all} matrices $(a_{ij})$ and yields a bound $K_G<\frac{\pi}{2\log\left(1+\sqrt{2}\right)}-\e_0$ for some effective constant $\e_0>0$. Other than disproving Krivine's conjecture, and along the way also disproving an intermediate conjecture of K\"onig that was made in 2000 as a step towards Krivine's conjecture, our main contribution is conceptual: despite
dealing with a binary rounding problem, random $2$-dimensional
projections combined with a  careful partition of $\R^2$ in order to
 round the projected vectors to values in $\{-1,1\}$, perform better than the ubiquitous random hyperplane technique. By establishing the usefulness of higher dimensional rounding schemes, this fact has consequences in approximation algorithms. Specifically, it yields the best known polynomial time approximation algorithm for the Frieze-Kannan Cut Norm problem, a
generic and well-studied optimization problem with many
applications.
\end{abstract}



\section{Introduction}

In his 1953 Resum\'e~\cite{Gro53}, Grothendieck proved a theorem
that he called ``le th\'eor\`eme fondamental de la th\'eorie
metrique des produits tensoriels". This result is known today as
Grothendieck's inequality. An equivalent formulation of
Grothendieck's inequality, due to Lindenstrauss and Pe\l
czy\'nski~\cite{LP68}, states that there exists a universal constant
$K\in (0,\infty)$ such that for every $m,n\in \N$, every $m\times n$
matrix $(a_{ij})$ with real entries, and every $m+n$ unit vectors
$x_1,\ldots,x_m,y_1,\ldots,y_n\in S^{m+n-1}$, there exist
$\e_1,\ldots,\e_m,\delta_1,\ldots,\delta_n\in \{-1,1\}$ satisfying
\begin{equation}\label{eq:gro}
\sum_{i=1}^m\sum_{j=1}^n a_{ij} \langle x_i,y_j\rangle \le K
\sum_{i=1}^m\sum_{j=1}^n a_{ij} \e_i\delta_j.
\end{equation}
Here $\langle\cdot,\cdot\rangle$ denotes the standard scalar product
on $\R^{m+n}$. The infimum over those $K\in (0,\infty)$ for
which~\eqref{eq:gro} holds true is called the Grothendieck constant,
and is denoted $K_G$.

 Grothendieck's
inequality is important to several disciplines, including the
geometry of Banach spaces, $C^*$ algebras, harmonic analysis,
operator spaces, quantum mechanics, and computer science. Rather
than attempting to explain the ramifications of Grothendieck's
inequality, we refer to the books~\cite{LP77,Sch81,Pis86,Jam87,DJT95,Ble01,Gar07,AK06,DFS08} and especially Pisier's recent
survey~\cite{Pis11}. The survey~\cite{KN11} is devoted
to Grothendieck's inequality in computer science;
Section~\ref{sec:comp} below contains a brief discussion of this
topic.

Problem 3 of Grothendieck's Resum\'e asks for the determination of
the exact value of $K_G$. This problem remains open despite major
effort by many mathematicians. In fact, even though $K_G$ occurs in
numerous mathematical theorems, and has equivalent interpretations
as a key quantity in physics~\cite{Tsi85,FR94} and computer
science~\cite{AN06,RS09}, we currently do not even know what is the
second digit of $K_G$; the best known bounds~\cite{Kri77,Ree91} are
$K_G\in(1.676,1.782)$.

Following the upper bounds on $K_G$ obtained
in~\cite{Gro53,LP68,Rie74} (see also the alternative proofs of~\eqref{eq:gro} in~\cite{Mau73,MP73,Pis78,Ble87,DJT95,JL01}, yielding worse bounds on $K_G$), progress on Grothendieck's Problem 3 halted after a
beautiful 1977 theorem of Krivine~\cite{Kri77}, who proved that
\begin{equation}\label{eq:krivine}
K_G\le \frac{\pi}{2\log\left(1+\sqrt{2}\right)}\ (=1.782...).
\end{equation}
One reason for this lack of improvement since 1977 is that Krivine
conjectured~\cite{Kri77} that his bound is actually the exact
value of $K_G$. Krivine's conjecture has become the leading conjecture in this area, and as such it has been restated repeatedly in subsequent publications; see for example~\cite{Kri79,Pis78,Pis86,Kon00,CHTW04}. The belief\footnote{E.g., quoting Pisier's book~\cite[p.~64]{Pis86},  ``The best known estimate for $K_G$ is due to Krivine, who proved (cf. [Kri3]) that $K_G\le \frac{2}{\log\left(1+\sqrt{2}\right)}=1.782...$ and conjectured that this is the exact value of $K_G$ . . .   Krivine claims that he checked $K_G>\pi/2$, and he has convincing (unpublished) evidence that his bound is sharp."} that the estimate~\eqref{eq:krivine} is the exact value of $K_G$ focused research on finding examples of matrices $(a_{ij})$ that exhibit a matching lower bound on $K_G$. Following work of Haagerup, Tomczak-Jaegermann and K\"onig, the search for such matrices led in 2000 to a clean intermediate conjecture of K\"onig~\cite{Kon00}, on maximizers of a certain oscillatory integral operator, that was shown to imply Krivine's conjecture; we will explain this conjecture, which we resolve in this paper, in Section~\ref{sec:intro kon} below. Here we prove that Krivine's conjecture is false,
thus obtaining the best known upper bound on $K_G$.
\begin{theorem}\label{thm:win!} There exists $\e_0>0$ such that
$$K_G<\frac{\pi}{2\log\left(1+\sqrt{2}\right)}-\e_0.$$
\end{theorem}
We stress that our proof is effective, and it readily yields a
concrete positive lower bound on $\e_0$. We chose not to state an
explicit new upper bound on the Grothendieck constant since we know that
our estimate is suboptimal. Section~\ref{sec:tiger} below contains a
discussion of potential improvements of our bound, based on
challenging open problems that conceivably might even lead to an
exact evaluation of $K_G$.

\begin{remark}\label{rem:complex}{\em
There has also been major effort to estimate the complex
Grothendieck constant~\cite{Gro53,Dav85,Pis78}; the best known upper
bound in this case is due to Haagerup~\cite{Haa87}. We did not
investigate this issue here, partly because  for complex scalars
there is no clean conjectured exact value of the Grothendieck
constant in the spirit of Krivine's conjecture. Nevertheless, it is
conceivable that our approach can improve Haagerup's bound on the
complex Grothendieck constant as well. We leave this research
direction open for future investigations.}
\end{remark}

In our opinion, the interest in the exact value of $K_G$ does not
necessarily arise from the importance of this constant itself,
though the reinterpretation of $K_G$ as a fundamental constant in
physics and computer science makes it even more interesting to know
 at least its first few digits. Rather, we believe that it is very
 interesting to understand the geometric configuration of unit
 vectors $x_1,\ldots,x_m,y_1,\ldots,y_n\in S^{m+n-1}$ (and matrix
 $a_{ij}$) which make the inequality~\eqref{eq:gro} ``most
 difficult". This issue is related to the ``rounding problem" in
 theoretical computer science; see Section~\ref{sec:comp}. With this
 in mind, Krivine's conjecture corresponds to a natural geometric
 intuition about the worst spherical configuration for Grothendieck's inequality.
 This geometric picture has been crystalized and cleanly formulated
 as an extremal analytic/geometric problem due to the works of Haagerup, K\"onig,
 and Tomczak-Jaegermann. We shall now explain this issue,
 since one of the main conceptual consequences of
 Theorem~\ref{thm:win!} is that the geometric picture behind Grothendieck's inequality that was
 previously believed to be true, is actually false. Along the way, we
 resolve a conjecture of K\"onig~\cite{Kon00}.

\subsection{K\"onig's problem}\label{sec:intro kon} One can reformulate Grothendieck's
inequality using integral operators (see~\cite{Kon00}). Given a
measure space $(\Omega,\mu)$ and a kernel $K\in L_1(\Omega\times
\Omega,\mu\times \mu)$, consider the integral operator
$T_K:L_\infty(\Omega,\mu)\to L_1(\Omega,\mu)$ induced by $K$, i.e.,
$$
T_Kf(x)\eqdef\int_{\Omega} f(y)K(x,y)d\mu(y).
$$
Grothendieck's inequality asserts that for every $f,g\in
L_\infty(\Omega,\mu;\ell_2)$, i.e., two bounded measurable functions
with values in  Hilbert space,
\begin{multline}\label{eq:operator form}
\int_\Omega\int_\Omega K(x,y)\langle f(x),g(y)\rangle
d\mu(x)d\mu(y)\\\le K_G \left\|T_K\right\|_{L_\infty(\Omega,\mu)\to
L_1(\Omega,\mu)}
\|g\|_{L_\infty(\Omega,\mu;\ell_2)}\|f\|_{L_\infty(\Omega,\mu;\ell_2)}.\end{multline}

K\"onig~\cite{Kon00}, citing unpublished computations of Haagerup,
asserts that the assumption
$K_G=\pi/\left(2\log\left(1+\sqrt{2}\right)\right)$ suggests that
the oscillatory Gaussian kernel $K:\R^n\times \R^n\to \R$ given by
\begin{equation}\label{eq:defK}
K(x,y)\eqdef
\exp\left(-\frac{\|x\|_2^2+\|y\|_2^2}{2}\right)\sin(\langle
x,y\rangle)
\end{equation}
should be extremal for Grothendieck's inequality in the asymptotic
sense, i.e., for $n\to \infty$. In the rest of this paper $K$ will
always stand for the kernel appearing in~\eqref{eq:defK}, and the
corresponding bilinear form $B_K:L_\infty(\R^n)\times
L_\infty(\R^n)\to \R$ will be given by
\begin{equation}\label{eq:B_k}
B_K(f,g)\eqdef\int_{\R^n}\int_{\R^n} f(x)g(y)K(x,y)dxdy.
\end{equation}
The above discussion led K\"onig to make the following conjecture:
\begin{conjecture}[K\"onig~\cite{Kon00}]\label{Q:konig} Define $f_0:\R^n\to \{-1,1\}$ by
$f_0(x_1,\ldots,x_n)=\sign(x_1)$. Then $B_K(f,g)\le B_K(f_0,f_0)$
for every $n\in \N$ and every measurable  $f,g:\R^n\to \{-1,1\}$.
\end{conjecture}

In~\cite{Kon00} the following result of K\"onig and
Tomczak-Jaegermann is proved:

\begin{proposition}[K\"onig and
Tomczak-Jaegermann~\cite{Kon00}]\label{prop:false} A positive answer
to Conjecture~\ref{Q:konig} would imply that
$K_G=\frac{\pi}{2\log\left(1+\sqrt{2}\right)}$.
\end{proposition}

Proposition~\ref{prop:false} itself can be viewed as motivation for
Conjecture~\ref{Q:konig}, since it is consistent with Haagerup's
work and Krivine's conjecture. But, there are additional reasons why
Conjecture~\ref{Q:konig} is natural. First of all, we know due to
Lieb's work~\cite{Lie90} that general Gaussian kernels, when viewed
as operators from $L_p(\R^n)$ to $L_q(\R^n)$, have only Gaussian
maximizers provided $p$ and $q$ satisfy certain conditions. The
kernel $K$ does not fit into Lieb's framework, since it is the
imaginary part of a Gaussian kernel (the Gaussian Fourier transform)
rather than an actual Gaussian kernel, and moreover the range
$p=\infty$ and $q=1$ is not covered by Lieb's theorem. Nevertheless,
in light of Lieb's theorem one might expect that maximizers of
kernels of this type have a simple structure, which could be viewed
as a weak justification of Conjecture~\ref{Q:konig}. A much more
substantial justification of Conjecture~\ref{Q:konig} is that
in~\cite{Kon00} K\"onig announced an unpublished result that he
obtained jointly with Tomczak-Jaegermann asserting that
Conjecture~\ref{Q:konig} is true for $n=1$.

\begin{theorem}\label{thm:KTJ}
For every Lebesgue measurable $f,g:\R\to \{-1,1\}$ we have
\begin{multline}\label{eq:KT}
\int_{\R}\int_\R f(x)g(y)\exp\left(-\frac{x^2+y^2}{2}\right)\sin(xy)dx dy\\\le
\int_{\R}\int_\R \sign(x)\sign(y)\exp\left(-\frac{x^2+y^2}{2}\right)\sin(xy)dx dy=2\sqrt{2}\log\left(1+\sqrt{2}\right).
\end{multline}
Moreover, equality in~\eqref{eq:KT} is attained only when $f(x)=g(x)=\sign(x)$ almost everywhere or $f(x)=g(x)=-\sign(x)$ almost everywhere.
\end{theorem}

We believe that it is important to have a published proof of
Theorem~\ref{thm:KTJ}, and for this reason we prove it in
Section~\ref{sec:KTJ}. Conceivably our proof is similar to the
unpublished proof of K\"onig and Tomczak-Jaegermann, though they
might have found a different explanation of this phenomenon. Since
Theorem~\ref{thm:win!} combined with Proposition~\ref{prop:false}
implies that K\"onig's conjecture is false, and as we shall see it
is false already for $n=2$, Theorem~\ref{thm:KTJ}  highlights special behavior of the one
dimensional case.

Our proof of Theorem~\ref{thm:win!} starts by disproving K\"onig's
conjecture for $n=2$. This is done in Section~\ref{sec:konig}.
Obtaining an improved upper bound on the Grothendieck constant
requires a substantial amount of additional work that uses the
counterexample to Conjecture~\ref{Q:konig}. This is carried out in
Section~\ref{sec:win}. The failure of K\"onig's conjecture shows
that the situation is more complicated than originally hoped, and in
particular that for $n>1$ the maximizers of the kernel $K$ have a
truly high-dimensional behavior. This more complicated geometric
picture  highlights the availability of high dimensional rounding
schemes that are more sophisticated (and  better) than ``hyperplane
rounding". These issues are discussed in Section~\ref{sec:comp} and
Section~\ref{sec:tiger}.

\section{Krivine-type rounding schemes and algorithmic
implications}\label{sec:comp}

Consider the following optimization problem. Given an $ m\times n$
matrix $A=(a_{ij})$, compute in polynomial time the value
\begin{equation}\label{eq:def opt}
\Opt(A)\eqdef \max_{\e_1,\ldots,\e_m,\delta_1,\ldots,\delta_n\in
\{-1,1\}} \sum_{i=1}^m\sum_{j=1}^n a_{ij}\e_i\delta_j.
\end{equation}
We refer to~\cite{AN06,KN11} for a discussion of the combinatorial
significance of this problem. It suffices to say here that it
relates to the problem of computing efficiently the Cut Norm of a
matrix, which is a subroutine in a variety of applications, starting
with the pioneering work of Frieze and Kannan~\cite{FK99}. Special
choices of matrices $A$ in~\eqref{eq:def opt} lead to specific
problems of interest, including efficient construction of
Szemer\'edi partitions~\cite{AN06}.

As shown in~\cite{AN06}, there exists $\delta_0>0$ such that the
existence of a polynomial time algorithm that outputs a number that
is guaranteed to be within a factor of $1+\delta_0$ of $\Opt(A)$
would imply that P=NP. But, since the quantity
$$
\SDP(A)\eqdef \max_{x_1,\ldots,x_m,y_1,\ldots,y_n\in S^{m+n-1}}
\sum_{i=1}^m\sum_{j=1}^m a_{ij} \langle x_i,y_j\rangle
$$
can be computed in polynomial time with arbitrarily good precision
(it is a semidefinite program~\cite{GLS81}), Grothendieck's inequality tells us
that the polynomial time algorithm that outputs the number $\SDP(A)$
is always within a factor of $K_G$ of $\Opt(A)$.

Remarkably, the work of Raghavendra and Steurer~\cite{RS09} shows
that $K_G$ has a complexity theoretic interpretation: it is likely
that no polynomial time algorithm can approximate $\Opt(A)$ to
within a factor smaller than $K_G$.
More precisely, it is shown in~\cite{RS09} that
$K_G$ is the Unique Games hardness threshold of the problem of
computing $\Opt(A)$. To explain what this means we briefly recall
Khot's Unique Games Conjecture~\cite{Kho02} (the version described
below is equivalent to the original one, as shown in~\cite{KKMO07}).

Khot's conjecture asserts that for every $\e>0$ there exists a prime
$p=p(\e)\in \N$ such that there is no polynomial time algorithm that,
given $n\in \N$ and a system of $m$ linear equations in $n$
variables of the form $$x_i-x_j\equiv c_{ij}\mod p$$ for some
$c_{ij}\in \N$, determines whether there exists an assignment of an
integer value to each variable $x_i$ such that at least $(1-\e)m$ of
the equations are satisfied, or whether no assignment of such values
can satisfy more than $\e m$ of the equations (if neither of these
possibilities occur, then an arbitrary output is allowed).

The Unique Games Conjecture is by now a common assumption that has
numerous applications in computational complexity. We have already
seen that there exists a polynomial time algorithm that computes
$\Opt(A)$ to within a factor of $K_G$. The Raghavendra-Steurer
theorem says that if there were a polynomial time algorithm ${ALG}$
that computes $\Opt(A)$ to within a factor $c<K_G$, then the Unique
Games Conjecture would be false. This means that there is
$\e=\e_c\in (0,1)$ such that, for all primes $p$, Raghavendra and
Steurer  design an algorithm that makes one call to
the algorithm ${ALG}$, with at most polynomially many additional
Turing machine steps, which successfully solves the problem described above on
the satisfiability of linear equations modulo $p$. Note that
Raghavendra and Steurer manage to do this despite the fact that the
value of $K_G$ is unknown.

Theorem~\ref{thm:win!} yields the first improved upper bound on the
Unique Games hardness threshold of the $\Opt(A)$ computation problem
since Krivine's 1977 bound. As we shall see, what hides behind
Theorem~\ref{thm:win!} is also a new algorithmic method which is of
independent interest. To explain this, note that the above
discussion dealt with the problem of computing the {\em number}
$\Opt(A)$. But it is actually of greater interest to find in
polynomial time signs $\e_1,\ldots,\e_m,\delta_1,\ldots,\delta_n\in
\{-1,1\}$ from among all such $2^{m+n}$ choices of signs, for which
$\sum_{i=1}^m\sum_{j=1}^n a_{ij}\e_i\delta_j$ is at least a constant
multiple $\Opt(A)$. This amounts to a ``rounding problem": we need to
find a procedure that, given
vectors $x_1,\ldots,x_m,y_1,\ldots,y_n\in S^{m+n-1}$, produces signs
$\e_1,\ldots,\e_m,\delta_1,\ldots,\delta_n\in \{-1,1\}$ whose
existence is ensured by Grothendieck's inequality~\eqref{eq:gro}.

Krivine's proof of~\eqref{eq:krivine} is based on a clever two-step
rounding procedure. We shall now describe a  generalization of
Krivine's method.

\begin{definition}[Krivine rounding scheme]\label{def:kri round}
Fix $k\in \N$ and assume that we are given two odd measurable functions $f,g:\R^k\to
\{-1,1\}$. Let $G_1,G_2\in \R^k$ be independent random vectors that
are distributed according to the standard Gaussian measure on $\R^k$, i.e., the measure with density $x\mapsto e^{-\|x\|_2^2/2}/(2\pi)^{k/2}$. For $t\in (-1,1)$ define
\begin{multline}\label{eq:h in intro}
H_{f,g}(t)\eqdef \E\left[f\left(\frac{1}{\sqrt{2}}G_1\right)g\left(\frac{t}{\sqrt{2}}G_1+\frac{\sqrt{1-t^2}}{\sqrt{2}}G_2\right)\right]\\
= \frac{1}{\pi^k(1-t^2)^{k/2}}\int_{\R^k}\int_{\R^k} f(x)g(y)\exp\left(\frac{-\|x\|_2^2-\|y\|_2^2+2t\langle x,y\rangle}{1-t^2}\right)dxdy.
\end{multline}
Then $H_{f,g}$ extends to an analytic function on the strip $\{z\in \C:\ \Re(z)\in (-1,1)\}$. We shall call $\{f,g\}$ a Krivine rounding scheme if $H_{f,g}$ is invertible on a neighborhood of the origin, and if we consider the Taylor expansion
\begin{equation}\label{eq:def coefficients}
H_{f,g}^{-1}(z)=\sum_{j=0}^\infty a_{2j+1} z^{2j+1}
\end{equation}
then there exists $c=c(f,g)\in (0,\infty)$ satisfying
\begin{equation}\label{eq:def c}
\sum_{j=0}^\infty |a_{2j+1}|c^{2j+1}=1.
\end{equation}
(Only odd Taylor coefficients appear in~\eqref{eq:def coefficients} since $H_{f,g}$, and therefore also $H_{f,g}^{-1}$, is odd.)
\end{definition}

\begin{definition}[Alternating Krivine rounding scheme]\label{def:kri round alt}
A Krivine rounding scheme $\{f,g\}$ is called an alternating Krivine rounding scheme if the coefficients $\{a_{2j+1}\}_{j=0}^\infty\subseteq \R$ in~\eqref{eq:def coefficients} satisfy $\sign(a_{2j+1})=(-1)^j$ for all $j\in \N\cup\{0\}$. Note that in this case equation~\eqref{eq:def c} becomes $H_{f,g}^{-1}(ic)/i=1$, or
\begin{equation}\label{eq:alternating}
c(f,g)=\frac{H_{f,g}(i)}{i}\stackrel{\eqref{eq:defK}\wedge\eqref{eq:B_k}\wedge\eqref{eq:h in intro}}{=} \frac{B_K(f,g)}{\left(\sqrt{2}\pi\right)^k}.
\end{equation}
\end{definition}

Given a Krivine rounding scheme $f,g:\R^k\to \{-1,1\}$ and
$x_1,\ldots,x_m,y_1,\ldots,y_n\in S^{m+n-1}$, the (generalized)
Krivine rounding method proceeds via the following two steps.

\medskip
\noindent{\bf Step 1 (preprocessing the vectors).} Consider the Hilbert space
$$
\H=\bigoplus_{j=0}^\infty \left(\R^{m+n}\right)^{\otimes (2j+1)}.
$$
For $x\in S^{m+n-1}$ we can then define two vectors $I(x),J(x)\in \H$ by
\begin{equation}\label{eq:IJ1}
I(x)\eqdef \sum_{j=0}^\infty |a_{2j+1}|^{1/2}c^{(2j+1)/2}x^{\otimes (2j+1)}
 \end{equation}
 and
 \begin{equation}\label{IJ2}
 J(x)\eqdef \sum_{j=0}^\infty \sign(a_{2j+1})|a_{2j+1}|^{1/2}c^{(2j+1)/2}x^{\otimes (2j+1)},
\end{equation}
where $c=c(f,g)$. The choice of $c$ was made in order to ensure that $I(x)$ and $J(x)$ are unit vectors in $\H$. Moreover, the definitions~\eqref{eq:IJ1} and~\eqref{IJ2} were made so that the following identity holds:
\begin{equation}\label{eq:identity gen groth}
\forall x,y\in S^{m+n-1},\quad \langle I(x),J(y)\rangle_\H=H_{f,g}^{-1}(c\langle x,y\rangle).
\end{equation}
The preprocessing step of the Krivine rounding method transforms the initial unit vectors $\{x_r\}_{r=1}^m,\{y_s\}_{s=1}^n\subseteq S^{m+n-1}$ to vectors $\{u_r\}_{r=1}^m,\{v_s\}_{s=1}^n\subseteq S^{m+n-1}$ satisfying the identities
\begin{equation}\label{eq:use gen groth iden}
\forall (r,s)\in \{1,\ldots,m\}\times \{1,\ldots,n\}\quad \langle u_r,v_s\rangle = \langle I(x_r),J(y_s)\rangle_\H\stackrel{\eqref{eq:identity gen groth}}{=}H_{f,g}^{-1}(c\langle x_r,y_s\rangle).
\end{equation}
As explained in~\cite{AN06}, these new vectors can be computed efficiently provided $H_{f,g}^{-1}$ can be computed efficiently; this simply amounts to computing a Cholesky decomposition.

\medskip
\noindent{\bf Step 2 (random projection).} Let $G: \R^{m+n}\to \R^k$ be a random $k\times (m+n)$ matrix whose entries are i.i.d.
standard Gaussian random variables. Define random signs $\sigma_1,\ldots,\sigma_m,\tau_1,\ldots,\tau_n\in \{-1,1\}$ by
\begin{equation}\label{eq:step 2}
\forall (r,s)\in \{1,\ldots,m\}\times \{1,\ldots,n\}\quad \sigma_r\eqdef f\left(\frac{1}{\sqrt{2}}Gu_r\right)\quad\mathrm{and}\quad \tau_s\eqdef g\left(\frac{1}{\sqrt{2}}Gv_s\right).
\end{equation}

\medskip

Having obtained the random signs $\sigma_1,\ldots,\sigma_m,\tau_1,\ldots,\tau_n\in \{-1,1\}$ as in~\eqref{eq:step 2}, for every $m\times n$ matrix $(a_{rs})$ we have
\begin{multline*}
\max_{\e_1,\ldots,\e_m,\delta_1,\ldots,\delta_n\in \{-1,1\}}\sum_{r=1}^m\sum_{s=1}^n a_{rs}\e_r\delta_s\ge \E\left[\sum_{r=1}^m\sum_{s=1}^n a_{rs}\sigma_r\tau_s\right]\\\stackrel{(\clubsuit)}{=} \E\left[\sum_{r=1}^m\sum_{s=1}^n a_{rs}
H_{f,g}\left(\langle u_r,v_s\rangle\right)\right]\stackrel{\eqref{eq:use gen groth iden}}{=}c(f,g)\sum_{r=1}^m\sum_{s=1}^n a_{rs}\langle x_r,y_s\rangle,
\end{multline*}
where ($\clubsuit$) follows by rotation invariance from~\eqref{eq:step 2} and~\eqref{eq:h in intro}. We have thus proved the following corollary, which yields a systematic way to bound the Grothendieck constant from above.

\begin{corollary}\label{cor:krivine}
Assume that $f,g:\R^k\to \{-1,1\}$ is a Krivine rounding scheme. Then
$$
K_G\le \frac{1}{c(f,g)}.
$$
\end{corollary}

Krivine's proof of~\eqref{eq:krivine} corresponds to Corollary~\ref{cor:krivine} when $k=1$ and $f(x)=g(x)=\sign(x)$. In this case $\{f,g\}$ is an alternating Krivine rounding scheme with $H_{f,g}(t)=\frac{2}{\pi}\arcsin(t)$ (Grothendieck's identity). By~\eqref{eq:alternating} we have $c(f,g)=\frac{2}{\pi i}\arcsin(i)=\frac{2}{\pi}\log\left(1+\sqrt{2}\right)$, so that Corollary~\ref{cor:krivine} does indeed correspond to Krivine's bound~\eqref{eq:krivine}.

One might expect that, since we want to round vectors $x_1,\ldots,x_m,y_1,\ldots,y_n\in S^{m+n-1}$ to signs $\e_1,\ldots,\e_m,\delta_1,\ldots,\delta_n\in \{-1,1\}$, the best possible Krivine rounding scheme occurs when $k=1$ and $f(x)=g(x)=\sign(x)$. This is the intuition leading to  K\"onig's conjecture.  The following simple corollary of Theorem~\ref{thm:KTJ} says that among all {\em one dimensional} Krivine rounding schemes $f,g:\R\to\{-1,1\}$ we indeed  have $c(f,g)\le c(\sign,\sign)$, so it does not pay off to take partitions of $\R$ which are more complicated than the half-line partitions.

\begin{lemma}\label{lem:cfg}
Let $f,g:\R\to \R$ be a Krivine rounding scheme. Then $c(f,g)\le \frac{2}{\pi}\log\left(1+\sqrt{2}\right)$.
\end{lemma}
\begin{proof}
Denote $c=c(f,g)$ and assume for contradiction that $c>\frac{2}{\pi}\log\left(1+\sqrt{2}\right)$. Let $r$ be the radius of convergence of the power series of $H_{f,g}^{-1}$ given in~\eqref{eq:def coefficients}. Due to~\eqref{eq:def c} we know that $r\ge c>\frac{2}{\pi}\log\left(1+\sqrt{2}\right)$. Denote
\begin{equation}\label{eq:def alpha rem}
\alpha\eqdef \frac{H_{f,g}(i)}{i}\stackrel{\eqref{eq:defK}\wedge\eqref{eq:B_k}\wedge\eqref{eq:h in intro}}{=}\frac{B_K(f,g)}{\pi\sqrt{2}}.
\end{equation}
By Theorem~\ref{thm:KTJ} we have $|\alpha|\le \frac{2}{\pi}\log\left(1+\sqrt{2}\right)<r$, and therefore $H_{f,g}^{-1}$ is well defined at the point $i\alpha\in \C$. Thus,
$$
1\stackrel{\eqref{eq:def alpha rem}}{=}\frac{H_{f,g}^{-1}(i\alpha)}{i}\stackrel{\eqref{eq:def coefficients}}{=} \sum_{j=0}^\infty (-1)^ja_{2j+1}\alpha^{2j+1}\le \sum_{j=0}^\infty |a_{2j+1}|\cdot |\alpha|^{2j+1}.
$$
By the definition of $c$ in~\eqref{eq:def c} we deduce that $c\le |\alpha|\le \frac{2}{\pi}\log\left(1+\sqrt{2}\right)$, as required.
\end{proof}

The conceptual message behind Theorem~\ref{thm:win!} is that, despite the above satisfactory state of affairs in the one dimensional case, it does pay off to use more complicated higher dimensional partitions. Specifically, our proof of Theorem~\ref{thm:win!} uses the following rounding procedure. Let $c,p\in (0,1)$ be small enough absolute constants. Given $\{x_r\}_{r=1}^m, \{y_s\}_{s=1}^n\subseteq  S^{m+n-1}$, we preprocess them to obtain new vectors $\{u_r=u_r(p,c)\}_{r=1}^m, \{v_s=v_s(p,c)\}_{s=1}^n\subseteq S^{m+n-1}$. Due to certain technical complications, these new vectors are obtained via a procedure that is similar to the preprocessing step (Step 1) described above, but is not identical to it. We refer to Section~\ref{sec:win} for a precise description of the preprocessing step that we use (we conjecture that this complication is unnecessary; see Conjecture~\ref{conj:no p}). Once the new vectors  $\{u_r\}_{r=1}^m, \{v_s\}_{s=1}^n\subseteq S^{m+n-1}$ have been constructed, we take an $2\times (m+n)$ matrix $G$ with entries that are i.i.d. standard Gaussian random variables, and we consider the random vectors $\{Gu_r=((Gu_r)_1,(Gu_r)_2)\}_{r=1}^m, \{Gv_s=((Gv_s)_1,(Gv_s)_2)\}_{s=1}^n\subseteq \R^2$. Having thus obtained new vectors in $\R^2$, with probability $(1-p)$ we ``round" our initial vectors to the signs $\{\sign((Gu_r)_2)\}_{r=1}^m, \{\sign((Gv_s)_2)\}_{s=1}^n\subseteq \R$, while with probability $p$ we round $x_r$  to $+1$ if
\begin{equation}\label{eq:fifth1}
(Gu_r)_2\ge c\left(((Gu_r)_1)^5-10((Gu_r)_1)^3+15(Gu_r)_1\right).
\end{equation}
and we round $x_r$ to $-1$ if
\begin{equation}\label{eq:fifth2}
(Gu_r)_2< c\left(((Gu_r)_1)^5-10((Gu_r)_1)^3+15(Gu_r)_1\right).
\end{equation}
For concreteness, at this juncture it suffices to describe our rounding procedure without explaining how it was derived --- the origin of the fifth degree polynomial appearing in~\eqref{eq:fifth1} and \eqref{eq:fifth2} will become clear in Section~\ref{sec:konig} and Section~\ref{sec:win}. The rounding procedure for $y_s$ is identical to~\eqref{eq:fifth1} and \eqref{eq:fifth2}, with $(Gv_s)_1, (Gv_s)_2$ replacing $(Gu_r)_1,(Gu_r)_2$, respectively.


\begin{figure}[here]
\begin{center}
\includegraphics[height=57mm, width=137mm]{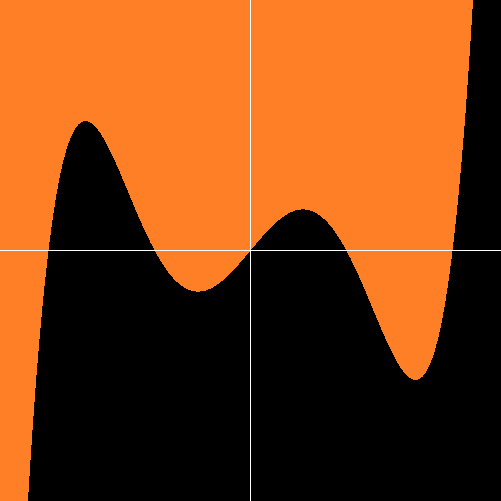}
\end{center}
\caption{The rounding procedure used in the proof of Theorem~\ref{thm:win!} relies on the partition of $\R^2$ depicted above. After a preprocessing step, high dimensional vectors are projected randomly onto $\R^2$ using a matrix with i.i.d. standard Gaussian entries. With a certain fixed probability, if the projected vector falls above the graph $y=c(x^5-10x^3+15x)$ then it is assigned the value $+1$, and otherwise it is assigned the value $-1$.}\label{fig:hermite}
\end{figure}

\section{The tiger partition and directions for future
research}\label{sec:tiger}

The partition of the plane described in Figure~\ref{fig:hermite} leads to a proof of Theorem~\ref{thm:win!}, but it is not the optimal partition for this purpose. It makes more sense to use the partitions corresponding to maximizers $f_{\max},g_{\max}:\R^2\to \{-1,1\}$ of Krivine's bilinear form $B_K$ as defined in~\eqref{eq:B_k}, i.e.,
\begin{multline}\label{eq:def maximizer}
B_K(f_{\max},g_{\max})=\max_{f,g:\R^2\to \{-1,1\}} B_K(f,g)\\=\max_{f,g:\R^2\to \{-1,1\}}\int_{\R^2}\int_{\R^2}f(x)g(y)\exp\left(-\frac{\|x\|_2^2+\|y\|_2^2}{2}\right)\sin(\langle x,y\rangle)dxdy.
\end{multline}
A straightforward weak compactness argument shows that the maximum in~\eqref{eq:def maximizer} is indeed attained (see Section~\ref{sec:konig}).

Given $f:\R^2\to \{-1,1\}$ define $\sigma(f):\R^2\to \{-1,1\}$ by
$$
\sigma(f)(y)\eqdef
\sign\left(\int_{\R^2}f(x)e^{-\|x\|_2^2/2}\sin\left(\langle
x,y\rangle \right)dx\right).
$$
Then
\begin{equation}\label{eq:fixed point}
\sigma(f_{\max})=g_{\max}\quad\mathrm{and}\quad  \sigma(g_{\max})=f_{\max}.
\end{equation}
Given $f:\R^2\to \{-1,1\}$ we can then hope to approach $f_{\max}$ by considering the iterates $\{\sigma^{2j}(f)\}_{j=1}^\infty$. If these iterates converge to $f_\infty$ then the pair of functions $\{f_\infty,\sigma(f_\infty)\}$ would satisfy the equations~\eqref{eq:fixed point}. One can easily check that $\sigma(f_0)=f_0$ when $f_0:\R^2\to \{-1,1\}$ is given by $f_0(x_1,x_2)=\sign(x_2)$. But, we have experimentally applied the above iteration procedure to a variety of initial functions $f\neq f_0$ (both deterministic and random choices), and in all cases the numerical computations suggest that the iterates $\{\sigma^{2j}(f)\}_{j=1}^\infty$ converge to the function $f_\infty$ that is depicted in Figure~\ref{fig:7} and Figure~\ref{fig:50} (the corresponding function $g_\infty=\sigma(f_\infty)$ is different from $f_\infty$, but has a similar structure).

\begin{figure}[here]
\begin{center}
\includegraphics[height=117mm]{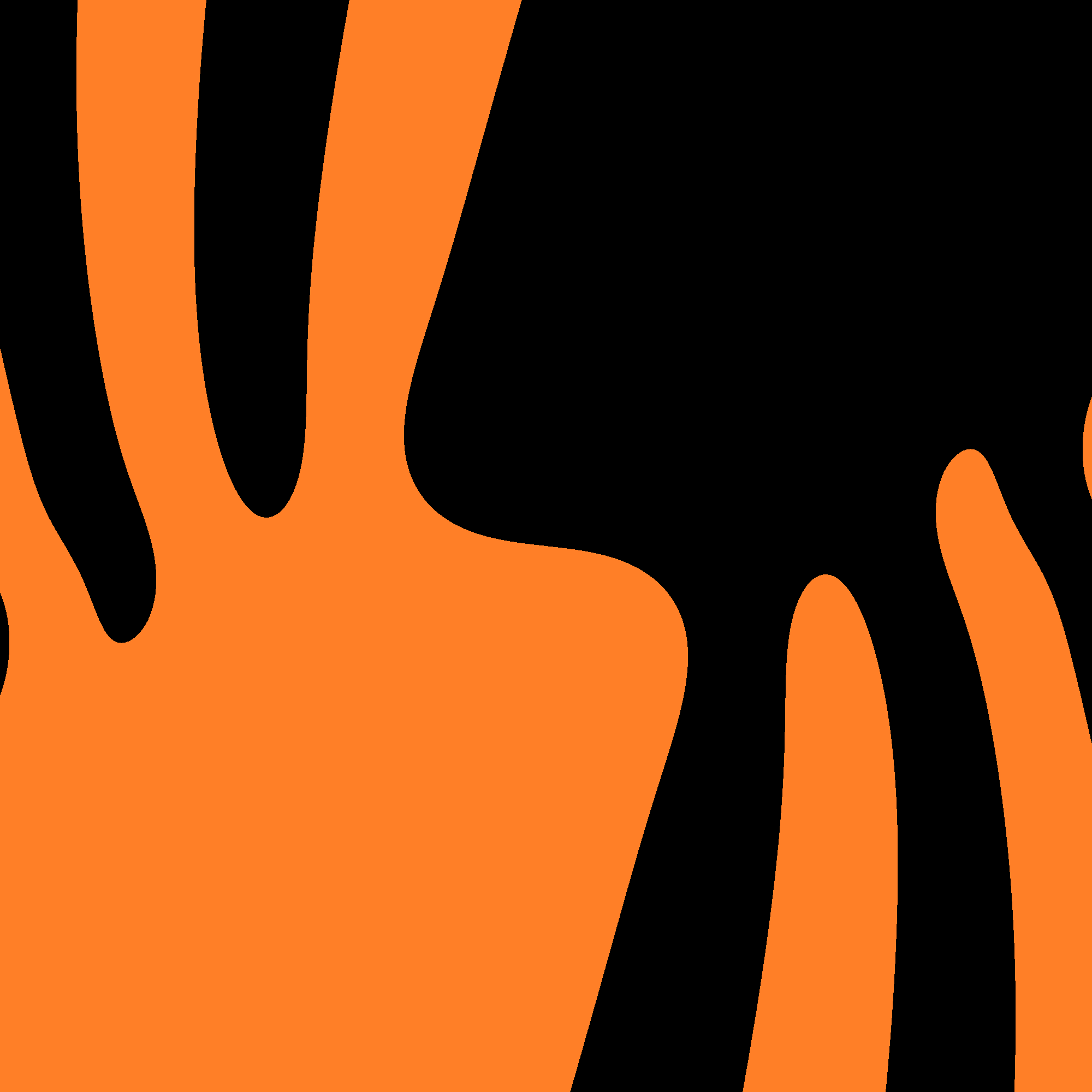}
\end{center}
\caption{The ``tiger partition": a depiction of the limiting function $f_\infty$ restricted to the square $[-7,7]\times [-7,7]\subseteq \R^2$, based on numerical computations. The two shaded regions correspond to the points where $f_\infty$ takes the values $+1$ and $-1$.}\label{fig:7}
\end{figure}

\begin{figure}[here]
\begin{center}
\includegraphics[height=117mm]{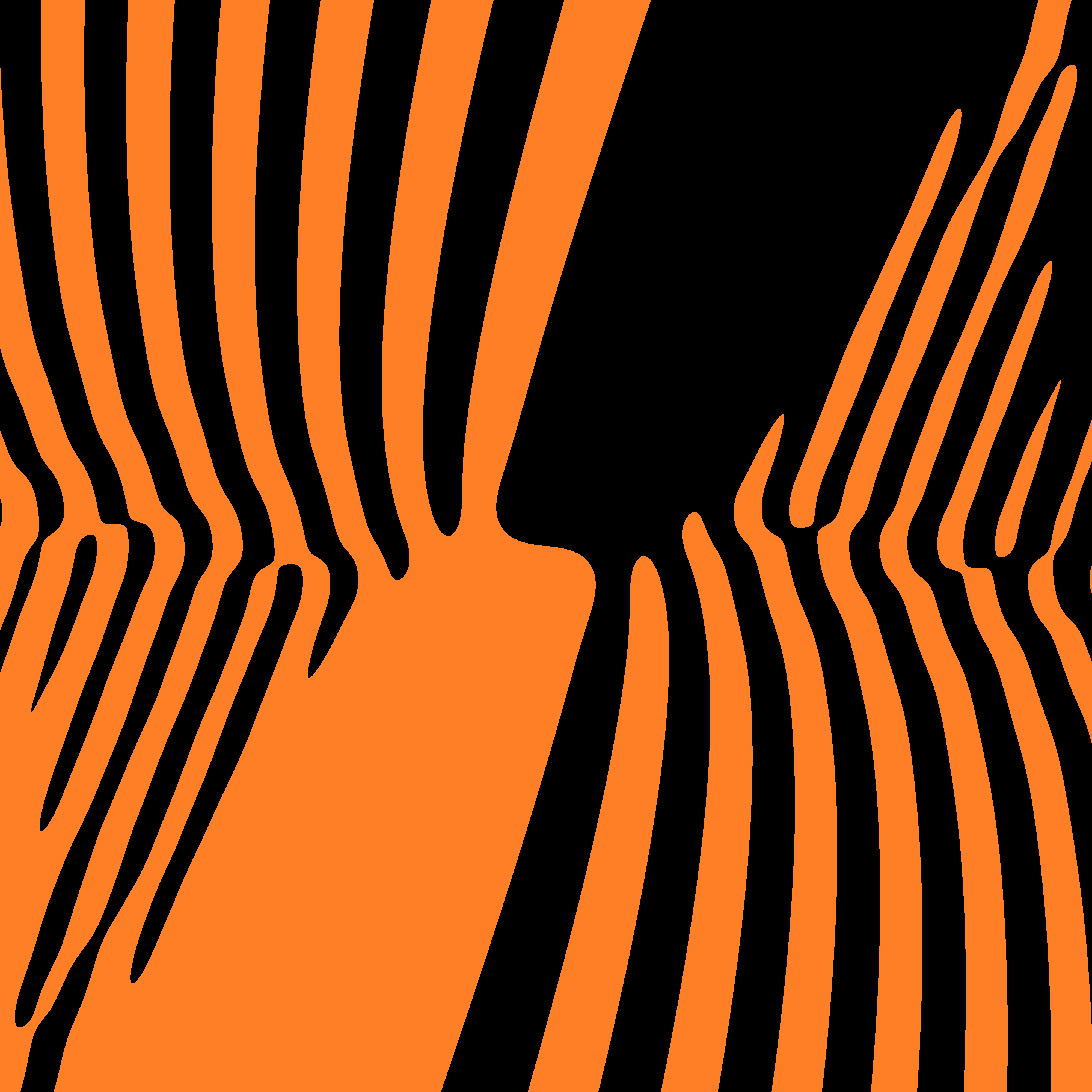}
\end{center}
\caption{A zoomed-out view of the tiger partition: a depiction of the limiting function $f_\infty$ restricted to the square $[-20,20]\times [-20,20]\subseteq \R^2$, based on numerical computations. }\label{fig:50}
\end{figure}

\begin{question}\label{Q:analytic}
Find an analytic description of the function $f_\infty$ from Figure~\ref{fig:7} and Figure~\ref{fig:50}. Our numerical computations suggest that the iterates $\{\sigma^{2j}(f)\}_{j=1}^\infty$ converge to $f_\infty$ for (almost?) all initial data $f:\R^2\to \{-1,1\}$ with $f\neq f_0$. Can this statement be made rigorous? If so, is it the case the $\{f_\infty,\sigma(f_\infty)\}$ are maximizers of the bilinear form $B_K$? We conjecture that the answer to this question is positive.
\end{question}

\begin{question}\label{Q:high dim}
Analogously to the above planar computations, can one find an analytic description of the maximizers $f_{\max}^{(n)},g_{\max}^{(n)}:\R^n\to \{-1,1\}$ of the $n$-dimensional version of K\"onig's bilinear form $B_K$? If so, does $\{f_{\max}^{(n)},g_{\max}^{(n)}\}$ form an alternating Krivine rounding scheme (recall Definition~\ref{def:kri round alt})?
\end{question}

We do not have sufficient data to conjecture whether the answer to Question~\ref{Q:high dim} is positive or negative. But, we note that if  $\{f_{\max}^{(n)},g_{\max}^{(n)}\}$ were an alternating Krivine rounding scheme then
\begin{equation}\label{eq:wild hope}
K_G=\sup_{n\in \N} \frac{\left(\sqrt{2}\pi\right)^n}{B_K\left(f_{\max}^{(n)},g_{\max}^{(n)}\right)}=\sup_{n\in \N} \frac{\left(\sqrt{2}\pi\right)^n}{\left\|T_K\right\|_{L_\infty(\R^n)\to L_1(\R^n)}}.
\end{equation}
Indeed, assuming that $\{f_{\max}^{(n)},g_{\max}^{(n)}\}$ is an alternating Krivine rounding scheme the upper bound in~\eqref{eq:wild hope} follows from Corollary~\ref{cor:krivine} and the identity~\eqref{eq:alternating}. For the reverse inequality in~\eqref{eq:wild hope} we proceed as in~\cite{Kon00}. Using~\eqref{eq:operator form} with $f,g:\R^n\to S^{n-1}$ given by $f(x)=g(x)=x/\|x\|_2$, we see that
\begin{equation}\label{eq:numerator}
K_G\ge \frac{\int_{\R^n}\int_{\R^n} K(x,y)\frac{\langle x,y\rangle}{\|x\|_2\|y\|_2}dxdy}{\left\|T_K\right\|_{L_\infty(\R^n)\to L_1(\R^n)}},
\end{equation}
and we conclude that~\eqref{eq:wild hope} is true since by equation (2.3) in~\cite{Kon00} the integral in the numerator of~\eqref{eq:numerator} equals $2^{n/2}\pi^n(1-1/n+O(1/n^2))$.

\section{A counterexample to K\"onig's conjecture}\label{sec:konig}

In this section we will make use of several facts on the Hermite
polynomials, for which we refer to~\cite[Sec.~6.1]{AAR99}. We let
$\{h_m:\R\to \R\}_{m=0}^\infty$ denote the sequence of Hermite
polynomials normalized so that they form an orthonormal basis with
respect to the measure on $\R$ whose density is $x\mapsto e^{-x^2}$.
Explicitly,
\begin{equation}\label{eq:def hermite}
h_m(x)\eqdef \frac{(-1)^m}{\sqrt{2^mm!\sqrt{\pi}}}\cdot
e^{x^2}\frac{d^m}{dx^m}\left(e^{-x^2}\right),
\end{equation}
so that
\begin{equation}\label{eq:orthogonality}
\int_{\R} h_m(x)h_k(x) e^{-x^2}dx=\delta_{mk}.
\end{equation}
For the correctness of the normalization in~\eqref{eq:def hermite},
see equation (6.1.5) in~\cite{AAR99}.

For reasons that will become clear in Remark~\ref{rem:alpha beta} below, we will consider
especially the fifth Hermite polynomial $h_5$, in which
case~\eqref{eq:def hermite} becomes
$$
h_5(x)=\frac{4x^5-20x^3+15x}{2\sqrt[4]{\pi}\sqrt{15}}.
$$
We record for future use the following technical lemma.

\begin{lemma}\label{lem:computations} The following identities hold true:
\begin{equation}\label{eq:fourth power}
\int_{-\infty}^\infty e^{-x^2}h_5(x)^4dx=\frac{4653}{\sqrt{\pi}},
\end{equation}
and
\begin{equation}\label{eq:mixed power}
\int_{-\infty}^\infty\int_{-\infty}^\infty \exp\left(-\frac{x_1^2+y_1^2}{2}\right)h_5(x_1)^2h_5(y_1)^2\cos(x_1y_1)dx_1dy_1 =49\sqrt{2}.
\end{equation}
\end{lemma}
\begin{proof} We shall use equation (6.1.7) in~\cite{AAR99}, which asserts that for every $x,t\in \R$ we have,
\begin{equation}\label{eq:generating}
\sum_{n=0}^\infty \frac{h_n(x)}{\sqrt{n!}}t^n=\frac{1}{\sqrt[4]{\pi}}\exp\left(\sqrt{2}tx-\frac{t^2}{2}\right).
\end{equation}
It follows that for every $x,u_1,u_2,u_3,u_4\in \R$,
\begin{multline}\label{eq:product1} \left(\sum_{a=0}^\infty
\frac{h_a(x)}{\sqrt{a!}}u_1^a\right)\left(\sum_{b=0}^\infty
\frac{h_b(x)}{\sqrt{b!}}u_2^b\right)
\left(\sum_{c=0}^\infty\frac{h_c(x)}{\sqrt{c!}}u_3^c\right)\left(\sum_{d=0}^\infty
\frac{h_d(x)}{\sqrt{d!}}u_4^d\right)\\=\frac{1}{\pi}
\exp\left(\sqrt{2} x(u_1+u_2+u_3+u_4)-\frac{u_1^2+u_2^2+u_3^2+u_4^2}{2}\right).
\end{multline}
Note that for every $A\in \R$ we have
\begin{equation}\label{eq:moment generating}
\int_{-\infty}^\infty \exp\left(-x^2+\sqrt{2}Ax\right)dx=e^{A^2/2}\int_{-\infty}^\infty \exp\left(-\left(x-\frac{A}{\sqrt{2}}\right)^2\right)dx=\sqrt{\pi}e^{A^2/2}.
\end{equation}
Hence, if we multiply both sides of~\eqref{eq:product1} by $e^{-x^2}$ and integrate over $x\in \R$, we obtain the following identity:
 \begin{multline}\label{eq:to equat coef1}
\sum_{a=0}^\infty\sum_{b=0}^\infty\sum_{c=0}^\infty\sum_{d=0}^\infty\frac{u_1^au_2^bu_3^cu_4^d}{\sqrt{a!b!c!d!}}
\int_{-\infty}^\infty\int_{-\infty}^\infty e^{-x^2}h_a(x)h_b(x)h_c(x)h_d(x)dx\\
\stackrel{\eqref{eq:product1}\wedge \eqref{eq:moment generating}}{=}\frac{1}{\sqrt\pi}\exp\left(\frac{(u_1+u_2+u_3+u_4)^2-u_1^2-u_2^2-u_3^2-u_4^2}{2}\right).
\end{multline}
By equating the coefficients of $u_1^5u_2^5u_3^5u_4^5$ on both sides of~\eqref{eq:to equat coef1}, we deduce that
$$
\frac{1}{14400}\int_{-\infty}^\infty e^{-x^2/2}h_5(x)^4dx=\frac{517}{1600\sqrt{\pi}},
$$
implying~\eqref{eq:fourth power}.

To prove~\eqref{eq:mixed power}, starting from~\eqref{eq:generating} we see that for every $x_1,y_1,u_1,u_2,v_1,v_2\in \R$ we have,
\begin{multline}\label{eq:product2} \left(\sum_{a=0}^\infty
\frac{h_a(x_1)}{\sqrt{a!}}u_1^a\right)\left(\sum_{b=0}^\infty
\frac{h_b(x_1)}{\sqrt{b!}}u_2^b\right)
\left(\sum_{c=0}^\infty\frac{h_c(y_1)}{\sqrt{c!}}v_1^c\right)\left(\sum_{d=0}^\infty
\frac{h_d(y_1)}{\sqrt{d!}}v_2^d\right)\\=\frac{1}{\pi}
\exp\left(\sqrt{2} x_1(u_1+u_2)+\sqrt{2} y_1(
v_1+v_2)-\frac{u_1^2+u_2^2+v_1^2+v_2^2}{2}\right).
\end{multline}
We next note the following identity: for every $A,B\in \R$ we have,
\begin{equation}\label{eq:AB2}
\int_{-\infty}^\infty\int_{-\infty}^\infty \exp\left(\sqrt{2}Ax+\sqrt{2}By-\frac{x^2+y^2}{2}\right)\cos(xy)dxdy=
\sqrt{2}\pi e^{(A^2+B^2)/2}\cos(AB).
\end{equation}
We will prove~\eqref{eq:AB2} in a moment, but let's first assume its validity and see how to complete the proof of~\eqref{eq:mixed power}. If we multiply~\eqref{eq:product2} by $\exp\left(-\frac{x_1^2+x_2^2}{2}\right)\cos(x_1y_1)$, and then integrate with respect to $(x_1,y_1)\in \R^2$, we obtain the following identity:
\begin{multline}\label{eq;to equat coef}
\sum_{a=0}^\infty\sum_{b=0}^\infty\sum_{c=0}^\infty\sum_{d=0}^\infty\frac{u_1^au_2^bv_1^cv_2^d}{\sqrt{a!b!c!d!}}
\int_{-\infty}^\infty\int_{-\infty}^\infty h_a(x_1)h_b(x_1)h_c(y_1)h_d(y_1)\cos(x_1y_1)dx_1dy_1\\
\stackrel{\eqref{eq:product2}\wedge \eqref{eq:AB2}}{=}\sqrt{2}e^{u_1u_2+v_1v_2}\cos((u_1+u_2)(v_1+v_2)).
\end{multline}
By equating the coefficients of $u_1^5u_2^5v_1^5v_2^5$ on both sides of~\eqref{eq;to equat coef}, we deduce that
$$
\frac{1}{14400}\int_{-\infty}^\infty\int_{-\infty}^\infty \exp\left(-\frac{x_1^2+y_1^2}{2}\right)h_5(x_1)^2h_5(y_1)^2\cos(x_1y_1)dx_1dy_1 =\frac{49\sqrt{2}}{14400},
$$
implying~\eqref{eq:mixed power}.

It remains to prove~\eqref{eq:AB2}. Let $I$ denote the integral on the right hand side of~\eqref{eq:AB2}. The change of variable $u=x-\sqrt{2}A$ and $v=y-\sqrt{2}B$ yields the following identity:
\begin{equation}\label{eq:I}
I= e^{A^2+B^2}\int_{-\infty}^\infty\int_{-\infty}^\infty \exp\left(-\frac{u^2+v^2}{2}\right)
\cos\left(\left(u+\sqrt{2}A\right)\left(v+\sqrt{2}B\right)\right)dudv.
\end{equation}
Note that
\begin{multline}\label{eq:trig}
\cos\left(\left(u+\sqrt{2}A\right)\left(v+\sqrt{2}B\right)\right)=
\cos\left(u\left(v+\sqrt{2}B\right)\right)\cos\left(\sqrt{2}A\left(v+\sqrt{2}B\right)\right)\\-
\sin\left(u\left(v+\sqrt{2}B\right)\right)\sin\left(\sqrt{2}A\left(v+\sqrt{2}B\right)\right).
\end{multline}
Since $\int_{-\infty}^\infty e^{-u^2/2}\sin\left(u\left(v+\sqrt{2}B\right)\right)du=0$, it follows that
\begin{multline}\label{eq:two cos}
e^{-A^2-B^2}I\\\stackrel{\eqref{eq:I}\wedge\eqref{eq:trig}}{=}
\int_{-\infty}^\infty\int_{-\infty}^\infty \exp\left(-\frac{u^2+v^2}{2}\right)
\cos\left(u\left(v+\sqrt{2}B\right)\right)\cos\left(\sqrt{2}A\left(v+\sqrt{2}B\right)\right)dudv.
\end{multline}
Since (see equation (6.1.1) in~\cite{AAR99}) for all $x_1\in \R$ we
have
\begin{equation}\label{eq:fourier gaussian}
\int_{-\infty}^\infty e^{-y_1^2/2}\cos(x_1y_1)dy_1=\sqrt{2\pi}
e^{-x_1^2/2},
\end{equation}
equation~\eqref{eq:two cos} becomes
\begin{multline}\label{eq:I2}
e^{-A^2-B^2}I=\sqrt{2\pi}\int_{-\infty}^\infty
\exp\left(-\frac{v^2+\left(v+\sqrt{2}B\right)^2}{2}\right)\cos\left(\sqrt{2}A\left(v+\sqrt{2}B\right)\right)dv\\
= \sqrt{2\pi}\int_{-\infty}^\infty
\exp\left(-\frac{\left(\sqrt{2}v+B\right)^2+B^2}{2}\right)\cos\left(\sqrt{2}A\left(v+\sqrt{2}B\right)\right)dv.
\end{multline}
The change of variable $w=\sqrt{2}c+B$ in~\eqref{eq:I2} gives,
\begin{eqnarray*}
e^{-A^2-B^2}I&=&\sqrt{\pi}e^{-B^2/2}\int_{-\infty}^\infty e^{-w^2/2}\cos\left(A\left(w-B\right)\right)dw\\
&=& \sqrt{\pi}e^{-B^2/2}\int_{-\infty}^\infty e^{-w^2/2}\left(\cos\left(Aw\right)\cos(AB)+\sin(Aw)\sin(AB)\right)dw\\&\stackrel{\eqref{eq:fourier gaussian}}{=}&
\sqrt{2}\pi e^{-(A^2+B^2)/2}\cos(AB).
\end{eqnarray*}
This concludes the proof of~\eqref{eq:AB2}, and therefore the proof of Lemma~\ref{lem:computations} is complete.
\end{proof}

For $\eta\in (0,1)$ let $f_\eta:\R^2\to \{-1,1\}$ be given by
\begin{equation}\label{eq:def f eta}
f_\eta(x_1,x_2)\eqdef\left\{\begin{array}{ll}1& x_2\ge \eta
h_5(x_1),\\-1&x_2< \eta h_5(x_1).\end{array}\right.
\end{equation}
Note that since $h_5$ is odd, so is $f_\eta$ (almost surely). For $z\in \C$
with $|\Re(z)|<1$ we define
\begin{equation}\label{eq:def H eta}
H_\eta(z)\eqdef\frac{1}{2\pi(1-z^2)}\int_{\R^2\times
\R^2}f_\eta(x)f_\eta(y)\exp\left(\frac{-\|x\|_2^2-\|y\|_2^2+2z\langle
x,y\rangle}{1-z^2}\right)dxdy.
\end{equation}
\begin{lemma}\label{lem:analytic}
 $H_\eta$ is analytic on the strip
\begin{equation}\label{eq:defS}
\S\eqdef \left\{z\in \C:\ |\Re(z)|<1\right\}.\end{equation}
Moreover, for all $a+bi\in \S$ we have
\begin{equation}\label{eq:ab}
|H_\eta(a+bi)|\le
\frac{\pi\left((1+a)^2+b^2\right)\left((1-a)^2+b^2\right)}{2(1-a^2)\sqrt{(1-a^2)^2+b^4+2(1+a^2)b^2}}.
\end{equation}
\end{lemma}

\begin{proof}Assume that $a\in (-1,1)$ and $b\in \R$, and write $z=a+bi$. Then
\begin{eqnarray}\label{eq:4-dim}
&&\!\!\!\!\!\!\!\!\!\!\!\!\!\!\!\!\!\!|H_\eta(z)|\nonumber\\\nonumber&\le&
\frac{1}{2\pi|1-z^2|}\int_{\R^2\times\R^2}\exp\left(-\Re\left(\frac{1}{1-z^2}\right)\left(\|x\|_2^2+\|y\|_2^2\right)+
2\Re\left(\frac{z}{1-z^2}\right)\langle
x,y\rangle\right)dxdy\\\nonumber
&=&\frac{1}{2\pi|1-z^2|}\int_{\R^2\times\R^2}\exp\left(-\frac12\Re\left(\frac{1}{1+z}\right)\|x+y\|_2^2
-\frac12\Re\left(\frac{1}{1-z}\right)\|x-y\|_2^2\right)dxdy\\
&=&\frac{1}{8\pi|1-z^2|}\int_{\R^2\times\R^2}\exp\left(-\frac12\Re\left(\frac{1}{1+z}\right)\|u\|_2^2
-\frac12\Re\left(\frac{1}{1-z}\right)\|v\|_2^2\right)dudv\\&=&
\frac{1}{8\pi|1-z^2|}\cdot\frac{(2\pi)^2}{\Re\left(\frac{1}{1-z}\right)\Re\left(\frac{1}{1+z}\right)}\nonumber
\\\label{eq:expression} &=&
\frac{\pi\left((1+a)^2+b^2\right)\left((1-a)^2+b^2\right)}{2(1-a^2)\sqrt{(1-a^2)^2+b^4+2(1+a^2)b^2}}
.
\end{eqnarray}
where in~\eqref{eq:4-dim} we used the change of variable $x+y=u$ and
$x-y=v$, whose Jacobian equals $\frac14$. Since the integral
defining $H_\eta$ is absolutely convergent on $\S$, the claim
follows.
\end{proof}


\begin{lemma}\label{lem: sanity check}
For every $z\in \C$ with $|\Re(z)|<1$ we have $H_0(z)=\arcsin(z)$.
\end{lemma}

\begin{proof} It suffices to prove this for $z\in (0,1)$. Writing $x=(x_1,x_2),y=(y_1,y_2)\in \R^2$, we have
\begin{eqnarray}\label{eq:var change}
&&\!\!\!\!\!\!\!\!\!\!\!\!\!\!\!\!\!\!\!\!\!H_0(z)=\left(\frac{1}{2\pi(1-z^2)}\int_{\R\times
\R}\sign(x_2)\sign(y_2)\exp\left(\frac{-x_2^2-y_2^2+2z
x_2y_2}{1-z^2}\right)dx_1dy_1\right)\nonumber\\\nonumber
&&\quad\cdot\left(\int_{\R\times \R}\exp\left(\frac{-x_1^2-y_1^2+2z x_1y_1}{1-z^2}\right)dx_2dy_2\right)\\
&=& \frac{1-z^2}{2\pi} \left(\int_{\R\times \R}
\sign(u)\sign(v)e^{-u^2-v^2+2zuv}dudv\right)\left(\int_{\R\times \R}
e^{-(u-zv)^2-(1-z^2)v^2}dudv\right).
\end{eqnarray}
Now,
\begin{equation}\label{eq:easy int}
\int_{\R\times \R} e^{-(u-zv)^2-(1-z^2)v^2}dudv=\left(\int_\R
e^{-w^2}dw\right)\left(\int_\R
e^{-(1-z^2)v^2}dv\right)=\frac{\pi}{\sqrt{1-z^2}}.
\end{equation}
Define
$$
\psi(z)=\int_{\R\times \R}
\sign(u)\sign(v)e^{-u^2-v^2+2zuv}dudv=2\int_0^\infty\int_0^\infty
e^{-u^2-v^2}\left(e^{2zuv}-e^{-2zuv}\right)dudv.
$$
Passing to the polar coordinates $u=r\cos\theta$ and
$v=r\sin\theta$, and then making the change of variable
$r=\sqrt{s}$, we see that
\begin{eqnarray*}
\psi(z)&=&2\int_0^{\frac{\pi}{2}}\int_0^\infty
re^{-r^2}\left(e^{zr^2\sin(2\theta)}-e^{-zr^2\sin(2\theta)}\right)drd\theta\\&=&
\int_0^{\frac{\pi}{2}}\int_0^\infty
e^{-s}\left(e^{sz\sin(2\theta)}-e^{-sz\sin(2\theta)}\right)dsd\theta
\\&=&\int_0^{\frac{\pi}{2}} \left(\frac{1}{1-z\sin(2\theta)}-\frac{1}{1+z\sin(2\theta)}\right)d\theta\\&=&
\int_0^{\frac{\pi}{2}}
\frac{2z\sin(2\theta)}{1-z^2\sin^2(2\theta)}d\theta.
\end{eqnarray*}
Due to the identity
$$
\frac{d}{d\theta}\left(\frac{-1}{\sqrt{1-z^2}}
\arcsin\left(\frac{z\cos(2\theta)}{\sqrt{1-z^2\sin^2(2\theta)}}\right)\right)=\frac{2z\sin(2\theta)}
{1-z^2\sin^2(2\theta)},
$$
we conclude that
\begin{equation}\label{eq:psi compute}
\psi(z)=\frac{2\arcsin(z)}{\sqrt{1-z^2}}.
\end{equation}
Lemma~\ref{lem: sanity check} now follows from
substituting~\eqref{eq:easy int} and~\eqref{eq:psi compute}
into~\eqref{eq:var change}.
\end{proof}

\begin{theorem}\label{thm:konig}
There exists $\eta_0>0$ such that for all $\eta\in (0,\eta_0)$ we
have
$$
\frac{H_\eta(i)}{i}\in\left(
\log\left(1+\sqrt{2}\right),\infty\right).
$$
\end{theorem}

Theorem~\ref{thm:konig} implies that the answer to K\"onig's problem
is negative. Indeed,
$$
\frac{H_\eta(i)}{i}\stackrel{\eqref{eq:def H
eta}}{=}\frac{1}{4\pi}\int_{\R^2\times
\R^2}f_\eta(x)f_\eta(y)\exp\left(-\frac{\|x\|_2^2+\|y\|_2^2}{2}\right)\sin\left(\langle
x,y\rangle\right)dx dy =\frac{B_K(f_\eta,f_\eta)}{4\pi} .
$$
Since $\arcsin(i)=i\log\left(1+\sqrt{2}\right)$, it follows from
Lemma~\ref{lem: sanity check} and Theorem~\ref{thm:konig} that for
every $\eta\in (0,\eta_0)$ we have
$B_K(f_\eta,f_\eta)>B_K(f_0,f_0)$. Since $f_0(x_1,x_2)=\sign(x_2)$,
the claimed negative answer to K\"onig's problem follows.

\begin{proof}[Proof of Theorem~\ref{thm:konig}]
Define $\f(\eta)=4\pi H_\eta(i)/i$. The required result will follow
once we prove that
$
\f(\eta)=\f(0)+1600\sqrt{2}\eta^4+O\left(\eta^6\right)
$
as $\eta\to 0$. To this end, since $\f$ is even, it suffices to show that
$\f''(0)=0$ and $\f''''(0)=38400\sqrt{2}$.

Since $h_5$ is odd we have,
\begin{multline*}
\f(\eta)\\=4\int_{-\infty}^\infty\int_{-\infty}^\infty \int_{\eta
h_5(x_1)}^\infty\int_{\eta h_5(y_1)}^\infty
\exp\left(-\frac{x_1^2+x_2^2+y_1^2+y_2^2}{2}\right)\sin(x_1y_1+x_2y_2)dx_2dy_2dx_1dy_1.
\end{multline*}
Define $F:\R^2\to \R$ by
\begin{multline*}
F(u_1,u_2) \\\eqdef 4\int_{-\infty}^\infty\int_{-\infty}^\infty
\int_{u_1 h_5(x_1)}^\infty\int_{u_2 h_5(y_1)}^\infty
\exp\left(-\frac{x_1^2+x_2^2+y_1^2+y_2^2}{2}\right)\sin(x_1y_1+x_2y_2)dx_2dy_2dx_1dy_1,
\end{multline*}
so that $\f(\eta)=F(\eta,\eta)$. Since $F$ is symmetric, i.e.,
$F(u_1,u_2)=F(u_2,u_1)$, it follows that
\begin{equation}\label{eq:''}
\f''(0)=2\frac{\partial^2F}{\partial
u_1^2}(0,0)+2\frac{\partial^2F}{\partial u_1\partial u_2}(0,0),
\end{equation}
and
\begin{equation}\label{eq:''''}
\f''''(0)=2\frac{\partial^4F}{\partial
u_1^4}(0,0)+8\frac{\partial^4F}{\partial u_1\partial
u_2^3}(0,0)+6\frac{\partial^4F}{\partial u_1^2\partial u_2^2}(0,0).
\end{equation}
Now,
\begin{multline}\label{eq:pure der u1}
\frac{\partial F}{\partial
u_1}(u_1,u_2)=-4\int_{-\infty}^\infty\int_{-\infty}^\infty \int_{u_2
h_5(y_1)}^\infty\exp\left(-\frac{x_1^2+u_1^2h_5(x_1)^2+y_1^2+y_2^2}{2}\right)\\\cdot
h_5(x_1)\sin\left(x_1y_1+u_1 h_5(x_1) y_2\right)dy_2dx_1dy_1.
\end{multline}
By differentiation of~\eqref{eq:pure der u1} under the integral with
respect to $u_1$, we see that
\begin{eqnarray}\label{eq:2 der at 0}
\frac{\partial^2F}{\partial
u_1^2}(0,0)&=&-4\int_{-\infty}^\infty\int_{-\infty}^\infty
\int_{0}^\infty\exp\left(-\frac{x_1^2+y_1^2+y_2^2}{2}\right)h_5(x_1)^2y_2\cos(x_1y_1)dy_2dx_1dy_1\nonumber\\
&=&\nonumber-4\int_{-\infty}^\infty\int_{-\infty}^\infty
\exp\left(-\frac{x_1^2+y_1^2}{2}\right)h_5(x_1)^2\cos(x_1y_1)dx_1dy_1\\&\stackrel{\eqref{eq:fourier gaussian}}{=}&-4\sqrt{2\pi}\int_{-\infty}^\infty
e^{-x_1^2}h_5(x_1)^2dx_1\stackrel{\eqref{eq:orthogonality}}{=}
-4\sqrt{2\pi}.
\end{eqnarray}
By differentiation of~\eqref{eq:pure der u1} with respect to $u_2$
we see that
\begin{multline}\label{eq:mixed der}
\frac{\partial^2 F}{\partial u_1\partial
u_2}(u_1,u_2)=4\int_{-\infty}^\infty\int_{-\infty}^\infty
\exp\left(-\frac{x_1^2+u_1^2h_5(x_1)^2+y_1^2+u_2^2h_5(y_1)^2}{2}\right)\\\cdot
h_5(x_1)h_5(y_1)\sin\left(x_1y_1+u_1u_2 h_5(x_1)
h_5(x_2)\right)dx_1dy_1.
\end{multline}
Hence,
\begin{equation}\label{eq:sin hermite}
\frac{\partial^2 F}{\partial u_1\partial
u_2}(0,0)=4\int_{-\infty}^\infty\int_{-\infty}^\infty
\exp\left(-\frac{x_1^2+y_1^2}{2}\right)
h_5(x_1)h_5(y_1)\sin\left(x_1y_1\right)dx_1dy_1.
\end{equation}
By equation (6.1.15) of~\cite{AAR99} we have
\begin{equation}\label{eq:sin h5}
\int_{-\infty}^\infty
e^{-y_1^2/2}h_5(y_1)\sin(x_1y_1)dy_1=\sqrt{2\pi}
e^{-x_1^2/2}h_5(x_1).
\end{equation}
Hence,
\begin{equation}\label{computed''}
\frac{\partial^2 F}{\partial u_1\partial
u_2}(0,0)\stackrel{\eqref{eq:sin hermite}\wedge\eqref{eq:sin
h5}}{=}4\sqrt{2\pi}\int_{-\infty}^\infty
e^{-x_1^2}h_5(x_1)^2dx_1\stackrel{\eqref{eq:orthogonality}}{=}4\sqrt{2\pi}.
\end{equation}
By substituting~\eqref{eq:2 der at 0}
and~\eqref{computed''} into~\eqref{eq:''}, we see that $\f''(0)=0$.

We shall now proceed to compute $\f''''(0)$, using the identity~\eqref{eq:''''}. By differentiation of~\eqref{eq:pure der u1} under the integral with
respect to $u_1$ three times, we see that
\begin{eqnarray}\label{eq:fourth moment appears}
&&\!\!\!\!\!\!\!\!\!\!\!\!\!\!\!\nonumber\frac{\partial^4 F}{\partial u_1^4}(0,0)\\\nonumber&=&12\int_{-\infty}^\infty\int_{-\infty}^\infty\int_0^\infty
\exp\left(-\frac{x_1^2+y_1^2+y_2^2}{2}\right)\cos(x_1y_1)h_5(x_1)^4\left(y_2+\frac{y_2^3}{3}\right)dy_2dx_1dy_1
\\\nonumber&=&20 \int_{-\infty}^\infty\int_{-\infty}^\infty
\exp\left(-\frac{x_1^2+y_1^2}{2}\right)\cos(x_1y_1)h_5(x_1)^4dx_1dy_1\\
&\stackrel{\eqref{eq:fourier gaussian}}{=}&20\sqrt{2\pi}\int_{-\infty}^\infty e^{-x_1^2}h_5(x_1)^4 dx_1
\stackrel{\eqref{eq:fourth power}}{=}93060\sqrt{2}.
\end{eqnarray}
Now, by differentiation of~\eqref{eq:mixed der} under the integral with
respect to $u_2$ twice, we obtain the identity
\begin{multline}\label{eq:fourth again}
\frac{\partial^4F}{\partial u_1\partial
u_2^3}(0,0)=-4\int_{-\infty}^\infty\int_{-\infty}^\infty\exp\left(-\frac{x_1^2+y_1^2}{2}\right)
h_5(y_1)^3h_5(x_1)\sin(x_1y_1)dx_1dy_1\\
\stackrel{\eqref{eq:sin h5}}{=}-4\sqrt{2\pi}\int_{-\infty}^\infty e^{-y_1^2}h_5(y_1)^4dy_1
\stackrel{\eqref{eq:fourth power}}{=}-18612\sqrt{2}.
\end{multline}
Finally, by differentiation of~\eqref{eq:mixed der} under the integral once with
respect to $u_1$ and once with respect to $u_2$, we see that
\begin{equation}\label{eq:cos squares}
\frac{\partial^4F}{\partial u_1^2\partial u_2^2}(0,0)=4\int_{-\infty}^\infty\int_{-\infty}^\infty\exp\left(-\frac{x_1^2+y_1^2}{2}\right)h_5(x_1)^2h_5(y_1)^2
\cos(x_1y_1)dx_1dy_1\stackrel{\eqref{eq:mixed power}}{=}196\sqrt{2}.
\end{equation}
Hence,
\begin{equation*}
\f''''(0)\stackrel{\eqref{eq:''''}\wedge\eqref{eq:fourth moment appears}\wedge\eqref{eq:fourth again}\wedge\eqref{eq:cos squares}}{=} 38400\sqrt{2}. \qedhere
\end{equation*}
\end{proof}

\begin{remark}\label{rem:alpha beta}
{\em Clearly, we did not arrive at the above proof by guessing that the fifth Hermite polynomial $h_5$ is the correct choice in~\eqref{eq:def f eta}. We arrived at this choice as the simplest member of a general family of ways to perturb the function $(x_1,x_2)\mapsto \sign(x_2)$. Since carrying out the analysis of this perturbation procedure in full generality is quite tedious, we chose to present the shorter proof above. Neverthless, we would like to explain here how we arrived at the choice of $h_5$ in~\eqref{eq:def f eta}.

Fix two odd functions $\alpha,\beta:\R\to \R$,
and write their Hermite expansion as
$$
\alpha(x)=\sum_{k=0}^\infty a_k h_{2k+1}(x) \quad
\mathrm{and}\quad \beta(x)=\sum_{k=0}^\infty b_k h_{2k+1}(x).
$$
For $\eta>0$ define $f_\eta,g_\eta:\R^2\to \{-1,1\}$ by
\begin{equation}\label{eq:def feta geta}
f_{\eta}(x_1,x_2)= \left\{ \begin{array}{ll}1 & x_2\ge \eta\alpha(x_1),\\ -1
& x_2<\eta\alpha(x_1),\end{array}\right. \quad\mathrm{and}\quad
g_{\eta}(x_1,x_2)= \left\{ \begin{array}{ll}1 & x_2\ge \eta\beta(x_1),\\ -1
& x_2<\eta\beta(x_1).\end{array}\right.
\end{equation}
For $z\in \C$
with $|\Re(z)|<1$ define
\begin{equation*}
H_\eta(z)=\frac{1}{2\pi(1-z^2)}\int_{\R^2\times
\R^2}f_\eta(x)g_\eta(y)\exp\left(\frac{-\|x\|_2^2-\|y\|_2^2+2z\langle
x,y\rangle}{1-z^2}\right)dxdy.
\end{equation*}
Thus, our final choice~\eqref{eq:def H eta} corresponds to  $\alpha=\beta=h_5$.

As in the proof of Theorem~\ref{thm:konig}, define $\f(\eta)=4\pi H_\eta(i)/i$. In order to show that we have $\f(\eta)>\f(0)$ for small enough $\eta\in (0,1)$, we start by computing $\f''(0)$, which turns out to be given by the following formula:
\begin{equation}\label{eq:second der positive}
\f''(0) =
-4\sqrt{2\pi}\sum_{k=0}^\infty \left(a_k -(-1)^k b_k\right)^2\le 0.
\end{equation}
Therefore, since $\f$ is odd and hence its odd order derrivatives at $0$ vanish, in order for us to have a chance to prove that $\f(\eta)>\f(0)$ for small enough $\eta\in (0,1)$, we must have $\f''(0)=0$. Due to~\eqref{eq:second der positive}, in terms of Hermite coefficients this forces the constraints
\begin{equation}\label{eq:constraint}
\forall k\in \N\cup\{0\},\quad a_k=(-1)^k b_k.
\end{equation}

We can therefore at best hope that $\f(\eta)$ is greater than $\f(0)$ by a fourth order term, and we need to  compute $\f''''(0)$. This is possible to do, using some identities from~\cite{AGV82}. Denote for $a,b,c,d,k\in \mathbb N\cup\{0\}$,
$$L_k(a,b,c,d) \eqdef
\frac{\sqrt{2(2a+1)!(2b+1)!(2c+1)!(2d+1)!}}{(2k)!(a+b+1-k)!(c+d+1-k)!}\cdot
\binom{2k}{a-b+k}\binom{2k}{c-d+k},
$$
where we use the convention that $L_k(a,b,c,d)=0$
whenever there is a negative number in the above factorials or
binomial coefficients. A somewhat tedious computation that uses results from~\cite{AGV82} shows that if the constraint~\eqref{eq:constraint} is satisfied then
$$
\f''''(0)=8\sum_{\substack{(k,p,q,r,s)\in \mathbb N\cup\{0\}\\p+q+r+s\text{ is
even}}} a_p a_q a_r a_s L_k(p,q,r,s) \left(1 +
3(-1)^{k+r+s}\right).
$$

If, for simplicity, we want to make the choice $\alpha=\beta$, the
simplest solution of the constraints~\eqref{eq:constraint} comes
from taking $\alpha=\beta=h_{5}$ ($\alpha=\beta=h_1$ won't work
since then one can check that $\f(\eta)=\f(0)$ for all $\eta$).
Note, however, that $\alpha=-\beta=h_3$ would work here too. }
\end{remark}

\section{Proof that
$K_G<\frac{\pi}{2\log\left(1+\sqrt{2}\right)}$}\label{sec:win}

We will fix from now on some $\eta\in (0,\eta_0)$, where $\eta_0$ is as in Theorem~\ref{thm:konig}.
For $p\in [0,1]$ define
$$
F_p\eqdef (1-p)H_0+pH_\eta,
$$
where $H_\eta$ is as in~\eqref{eq:def H eta}.
In what follows we will denote the unit disc in $\C$ by
$$
\D\eqdef \left\{z\in \C:\  |z|< 1\right\}.
$$

\begin{theorem}\label{thm:abs}
The exists $p_0>0$ such that for all $p\in (0,p_0)$ we have
$F_p(\S)\supseteq \frac{9}{10}\D$ and $F_p^{-1}$ is well defined and
analytic on $\frac{9}{10}\D$. Moreover, if we write
$F_p^{-1}(z)=\sum_{k=1}^\infty a_k(p) z^k$ then there exists
$\gamma=\gamma_p\in [0,\infty)$ satisfying
\begin{equation}\label{eq:1}
\sum_{k=1}^\infty
|a_k(p)|\gamma^k=1,
\end{equation}
and
\begin{equation}\label{eq:c}
\gamma> \log\left(1+\sqrt{2}\right)=0.88137...
\end{equation}
\end{theorem}

Assuming Theorem~\ref{thm:abs} for the moment, we will now deduce
Theorem~\ref{thm:win!}.

\begin{proof}[Proof of Theorem~\ref{thm:win!}]
Fix $p\in (0,p_0)$ and let $\gamma>0$ be the constant from
Theorem~\ref{thm:abs}. Due to~\eqref{eq:1}, $\sum_{k=1}^\infty
a_k(p)\gamma^k$ converges absolutely, and therefore $F_p^{-1}$ is
analytic and well defined on $\gamma\D$. For small enough $p$ some of the coefficients
$\{a_k(p)\}_{k=1}^\infty$ are negative (since the third Taylor coefficient of $H_0^{-1}(z)=\sin z$ is negative),
implying that for every $r\in [0,1]$ we have
\begin{equation}\label{eq:in strip}
F_p^{-1}(r\gamma)=\sum_{k=1}^\infty a_k(p)r^k\gamma^k\in (-1,1)\subseteq \S.
\end{equation}

Let $\H$ be a Hilbert space. Define two mappings
$$
L_p,R_p:\H\to \bigoplus_{k=1}^\infty \H^{\otimes k}\eqdef \mathcal K
$$
by
$$
L_p(x)\eqdef \sum_{k=1}^\infty \sqrt{|a_k(p)|}\gamma^{k/2}x^{\otimes k}\quad\mathrm{and}\quad R_p(x)\eqdef \sum_{k=1}^\infty \sign(a_k(p))\sqrt{|a_k(p)|}\gamma^{k/2}x^{\otimes k}.
$$
By~\eqref{eq:1}, if $\|x\|_{\H}=1$ then
$\|L_p(x)\|_\K=\|R_p(x)\|_\K=1$. Moreover, if $\|x\|_\H=\|y\|_\H=1$
then
\begin{equation}\label{eqLdot prodoct in strip}
\langle L_p(x),R_p(y)\rangle =\sum_{k=1}^\infty
a_k(p)\gamma^k\langle x,y\rangle^k =F_p^{-1}(\gamma\langle
x,y\rangle)\stackrel{\eqref{eq:in strip}}{\in} \S.
\end{equation}

For $N\in \N$ let $G:\R^N\to \R^2$ be a $2\times N$ random matrix
with i.i.d. standard Gaussian entries. Let $g_1,g_2\in \R^2$ be the
first two columns of $G$ (i.e., $g_1,g_2$ are i.i.d. standard two
dimensional Gaussian vectors). If $x,y\in \R^N$ are unit vectors
satisfying $\langle x,y\rangle\in \S$ then by rotation invariance we
have
\begin{eqnarray}\label{eq:gen groth}
&&\!\!\!\!\!\!\!\!\!\!\E\left[f_\eta\left(\frac{1}{\sqrt{2}}Gx\right)f_\eta\left(\frac{1}{\sqrt{2}}Gy\right)\right]=
\E\left[f_\eta\left(\frac{g_1}{\sqrt{2}}\right)f_\eta\left(\langle
x,y\rangle \frac{g_1}{\sqrt{2}}+\sqrt{1-\langle
x,y\rangle^2}\frac{g_2}{\sqrt{2}}\right)\right]\nonumber\\&=&\frac{1}{(2\pi)^2}\int_{\R^2\times
\R^2}f_\eta\left(\frac{u}{\sqrt{2}}\right) f_\eta\left(\langle x,y\rangle \frac{u}{\sqrt{2}}+\sqrt{1-\langle
x,y\rangle^2}\frac{v}{\sqrt{2}}\right)\exp\left(-\frac{\|u\|_2^2+\|v\|_2^2}{2}\right)dudv\nonumber\\&=&\frac{2}{\pi}H_\eta(\langle
x,y\rangle),
\end{eqnarray}
where we made the change of variable $u=\sqrt{2}u'$ and $v=\left(\sqrt{2}v'-\sqrt2\langle x,y\rangle u'\right)/\sqrt{1-\langle x,y\rangle^2}$, whose Jacobian is $4/(1-\langle x,y\rangle^2)$.

Fix an $m\times n$ matrix $A=(a_{ij})$ and let
$x_1,\ldots,x_m,y_1,\ldots,y_n\in \H$ be unit vectors satisfying
\begin{equation}\label{eq:Mmax}
\sum_{i=1}^m\sum_{j=1}^n a_{ij}\langle x_i,y_j\rangle=M\eqdef
\max_{u_1,\ldots,u_m,v_1,\ldots,v_n\in S_\H}\sum_{i=1}^m\sum_{j=1}^n a_{ij}\langle u_i,v_j\rangle,
\end{equation}
where $S_\H$ denotes the unit sphere of $\H$. Consider the unit
vectors $\{L_p(x_i)\}_{i=1}^m\cup\{R_p(y_j)\}_{j=1}^n$, which we can
think of as residing in $\R^N$ for $N=m+n$. By~\eqref{eqLdot prodoct
in strip} we have $\langle L_p(x_i),R_p(y_j)\rangle\in \S$ for all
$i\in \{1,\ldots,m\}$ and $j\in \{1,\ldots,n\}$, so that we may use
the identity~\eqref{eq:gen groth} for these vectors. Let $\lambda$
be a random variable satisfying $\Pr[\lambda=1]=p$,
$\Pr[\lambda=0]=1-p$. Assume that $\lambda$ is independent of $G$.
Define random variables
$\e_1,\ldots,\e_m,\delta_1,\ldots,\delta_n\in \{-1,1\}$ by
$$
\e_i=(1-\lambda) f_0\left(\frac{1}{\sqrt{2}}G L_p(x_i)\right)+\lambda f_\eta\left(\frac{1}{\sqrt{2}}G
L_p(x_i)\right)
$$
and
$$
 \delta_j=(1-\lambda) f_0\left(\frac{1}{\sqrt{2}}G
R_p(y_j)\right)+\lambda f_\eta\left(\frac{1}{\sqrt{2}}G R_p(y_j)\right).
$$

Then,
\begin{eqnarray*}
&&\!\!\!\!\!\!\!\!\!\!\!\!\!\!\!\!\!\!\!\!\!\!\!\!\!\!\!\!\max_{\sigma_1,\ldots,\sigma_m,\tau_1,\ldots,\tau_n\in
\{-1,1\}}\sum_{i=1}^m\sum_{j=1}^n a_{ij}
\sigma_i\tau_j\ge\E\left[\sum_{i=1}^m\sum_{j=1}^n a_{ij}
\e_i\delta_j\right]\\&\stackrel{\eqref{eq:gen
groth}}{=}&\frac{2}{\pi}\sum_{i=1}^m\sum_{j=1}^n
a_{ij}\Big((1-p)H_0\big(\langle
L_p(x_i),R_p(y_j)\rangle\big)+pH_\eta\big(\langle
L_p(x_i),R_p(y_j)\rangle\big)\Big)\\&=&
\frac{2}{\pi}\sum_{i=1}^m\sum_{j=1}^n a_{ij}F_p\big(\langle
L_p(x_i),R_p(y_j)\rangle\big)\\
&\stackrel{\eqref{eqLdot prodoct in
strip}}{=}&\frac{2}{\pi}\sum_{i=1}^m\sum_{j=1}^n
a_{ij}F_p\left(F_p^{-1}\gamma\langle
x_i,y_j\rangle\right)\stackrel{\eqref{eq:Mmax}}=
\frac{2\gamma}{\pi}M.
\end{eqnarray*}
This gives the bound $K_G\le
\frac{\pi}{2\gamma}\stackrel{\eqref{eq:c}}{<}\frac{\pi}{2\log\left(1+\sqrt{2}\right)}$,
as required.
\end{proof}

Our goal from now on will be to prove Theorem~\ref{thm:abs}.

\begin{lemma}\label{lem:1-1}
$H_0$ is one-to-one on $\S$ and $H_0(\S)\supseteq \D$.
\end{lemma}
\begin{proof}
The fact that $H_0$ is one-to-one on $\S$ is a consequence of
Lemma~\ref{lem: sanity check}. To show that $H_0(\S)\supseteq \D$ we
need to prove that if $a,b\in \R$ and $a^2+b^2<1$ then
$|\Re\left(\sin\left(a+bi\right)\right)|<1$. Now,
\begin{equation}\label{eq:trig}
\left|\Re\left(\sin\left(a+bi\right)\right)\right|=\frac{e^{b}+e^{-b}}{2}|\sin a|.
\end{equation}
Using the inequality $|\sin a|\le|a|$, we see that it
suffices to show that  for all $x\in (0,1)$ we have
\begin{equation}\label{eq:root goal}
\frac{e^{x}+e^{-x}}{2}\sqrt{1-x^2}< 1
\end{equation}
By Taylor's formula we know that there exists $y\in [0,x]$ such that
\begin{equation}\label{eq:cosh}
\frac{e^x+e^{-x}}{2}=1+\frac{x^2}{2}+\frac{x^4}{24}\cdot\frac{e^{y}+e^{-y}}{2}\le
1+\frac{x^2}{2}+\frac{x^4}{24}\cdot\frac{e+e^{-1}}{2}<1+\frac{x^2}{2}+\frac{x^4}{12}.
\end{equation}
Note that
$$
\left(1+\frac{x^2}{2}+\frac{x^4}{12}\right)^2\left(1-x^2\right)=
1-\frac{7x^4}{12}-\frac{x^6}{3}-\frac{11x^8}{144}-\frac{x^{10}}{144}< 1,
$$
which together with~\eqref{eq:cosh} implies~\eqref{eq:root goal}.
\end{proof}

\begin{remark}
{\em A more careful analysis of the expression~\eqref{eq:trig} shows that
there exists $\e_0>0$ (e.g., $\e_0=0.05$ works) such that
$H_0(\S)\supseteq (1+\e_0)\D$. Since we will not need this stronger
fact here, we included the above simpler proof of a weaker statement.}
\end{remark}

\begin{lemma}\label{lem:rouche}
For every $r\in (0,1)$ there exists $p_r\in (0,1)$ and
a bounded open subset $\Omega_r\subseteq \S$ with
$\overline{\Omega_r}\subseteq \S$ such that for all $p\in (0,p_r)$
the function $F_p$ is one-to-one on $\Omega_r$ and $F_p(\Omega_r)=
r\D$. Thus $F_p^{-1}$ is well defined and analytic on $r\D$.
\end{lemma}

\begin{proof}
For $n\in \N$ consider the set
$$
E_n=\left\{z\in \C:\ |\Re(z)|< 1-\frac{1}{n}\ \wedge\ |\Im(z)|<
n\right\}.
$$
Using Lemma~\ref{lem:1-1}, fix a large enough $n\in \N$ so that
$H_0(E_n)\supseteq r\overline{\D}$.  The bound~\eqref{eq:ab} implies that there
exists $M>0$ such that $|H_\eta(z)|\le M$ for all $\eta>0$ and $z\in
\partial E_{n+1}$.  By Lemma~\ref{lem:1-1}, $H_0$ takes a value
$\zeta\in r\overline{\D}$ exactly once on $E_{n+1}$, and this occurs at some point in
$E_n$. Hence,
$$
m\eqdef\min_{\stackrel{ \zeta\in r\D}{z\in \partial E_{n+1}}} |H_0(z)-\zeta|>0.
$$
Define $p_r=m/(2M)$.

Fix $\zeta\in r\D$. If $p\in (0,p_r)$ then for every $z\in \partial
E_{n+1}$ we have
$$
\left|p\left(H_\eta(z)-H_0(z)\right)\right|<\frac{m}{2M}\left(|H_\eta(z)|+|H_0(z)|\right)\le
m\le \left|H_0(z)-\zeta\right|.
$$
Rouch\'e's theorem now implies that the number of zeros of
$H_0-\zeta$ in $E_{n+1}$ is the same as the number of zeros of
$H_0-\zeta+p\left(H_\eta-H_0\right)=F_p-\zeta$ in $E_{n+1}$. Hence
$F_p$ takes the value $\zeta$ exactly once in $E_{n+1}$. Since
$\zeta$ was an arbitrary point in $r\D$, we can define
$\Omega_r=F_p^{-1}(r\D)$.
\end{proof}

\begin{lemma}\label{lem:bound for cauchy} For every $r\in (0,1)$
there exists $C_r\in (0,\infty)$ such that, using the notation of
Lemma~\ref{lem:rouche}, for every $p\in (0,p_r)$ and $z\in r\D$ we
have
\begin{equation}\label{eq:perturbed taylor}
\left|F_p^{-1}(z)-\sin z-p\left(z-H_\eta(\sin z)\right)\cos
z\right|\le C_rp^2.
\end{equation}
\end{lemma}

\begin{proof}
Note that
\begin{equation}\label{eq:for diff}
z=F_p\left(F_p^{-1}(z)\right)=(1-p)H_0\left(F_p^{-1}(z)\right)+pH_\eta\left(F_p^{-1}(z)\right).
\end{equation}
By differentiating~\eqref{eq:for diff} with respect to $p$, we see
that
\begin{eqnarray*}
0&=&H_\eta\left(F_p^{-1}(z)\right)-H_0\left(F_p^{-1}(z)\right)\\&&\quad+\left(\frac{d}{dp}F_p^{-1}(z)\right)
\left((1-p)\frac{d H_0}{dz}\left(F_p^{-1}(z)\right)+p\frac{dH_\eta}{dz}\left(F_p^{-1}(z)\right)\right)\\
&=&H_\eta\left(F_p^{-1}(z)\right)-H_0\left(F_p^{-1}(z)\right)+\left(\frac{d}{dp}F_p^{-1}(z)\right)
\frac{dF_p}{dz}\left(F_p^{-1}(z)\right)\\&=&H_\eta\left(F_p^{-1}(z)\right)-H_0\left(F_p^{-1}(z)\right)
+\frac{\frac{d}{dp}F_p^{-1}(z)}{\frac{d}{dz}\left(F_p^{-1}(z)\right)}.
\end{eqnarray*}
Hence,
\begin{equation}\label{eq:first der}
\frac{d}{dp}F_p^{-1}(z)=\left(H_0\left(F_p^{-1}(z)\right)-H_\eta\left(F_p^{-1}(z)\right)\right)
\frac{d}{dz}\left(F_p^{-1}(z)\right).
\end{equation}
If we now differentiate~\eqref{eq:first der} with respect to $p$,
while using~\eqref{eq:first der} whenever the term
$\frac{d}{dp}F_p^{-1}(z)$ appears,  we obtain the following
identity.
\begin{eqnarray}\label{eq:seconf der}
&&\!\!\!\!\!\!\!\!\!\!\!\!\!\!\!\!\!\!\!\!\!\nonumber\frac{d^2}{dp^2}F_p^{-1}(z)\\\nonumber&=&\left[\frac{d
H_0}{dz}\left(F_p^{-1}(z)\right)-\frac{d
H_\eta}{dz}\left(F_p^{-1}(z)\right)+\left(H_0\left(F_p^{-1}(z)\right)-H_\eta\left(F_p^{-1}(z)\right)\right)
\frac{d^2}{dz^2}\left(F_p^{-1}(z)\right)\right]\\
&&\quad\cdot
\left(H_0\left(F_p^{-1}(z)\right)-H_\eta\left(F_p^{-1}(z)\right)\right)
\frac{d}{dz}\left(F_p^{-1}(z)\right).
\end{eqnarray}
Take $M=M_r>0$ such that for all $w\in \Omega_r$ we have
\begin{equation}\label{eq:M}
\max\left\{|H_0(w)|,|H_\eta(w)|,\left|\frac{dH_0}{dz}(w)\right|,\left|\frac{dH_\eta}{dz}(w)\right|\right\}\le
M.
\end{equation}
Note that~\eqref{eq:M} applies to $w=F_p^{-1}(z)$ for $z\in r\D$. We
also define
$$
R=R_r=\max_{w\in \partial \Omega_{(1+r)/2}} |w|.
$$
Then for $\zeta\in \frac{1+r}{2}\D$ we have
$\left|F_p^{-1}(\zeta)\right|\le R$. If $z\in r\D$ then by the
Cauchy formula we have
$$
\left|\frac{d}{dz}\left(F_p^{-1}(z)\right)\right|=\left|\frac{1}{\pi(r+1)i}\oint_{\frac{1+r}{2}\partial
\D}\frac{F_p^{-1}(\zeta)}{(\zeta-z)^2}d\zeta\right|\le
\max_{\zeta\in \frac{1+r}{2}\partial
\D}\frac{\left|F_p^{-1}(\zeta)\right|}{(|\zeta|-|z|)^2}\le
\frac{4R}{(1-r)^2}.
$$
Similarly,
$$
\left|\frac{d^2}{dz^2}F_p^{-1}(z)\right|=\left|\frac{2}{\pi(r+1)i}\oint_{\frac{1+r}{2}\partial
\D}\frac{F_p^{-1}(\zeta)}{(\zeta-z)^3}d\zeta\right|\le
\frac{16R}{(1-r)^3}.
$$
These estimates, in conjunction with the identity~\eqref{eq:seconf
der}, imply the following bound:
$$
\left|\frac{d^2}{dp^2}F_p^{-1}(z)\right|\le\left(2M+2M\frac{16R}{(1-r)^3}\right)2M\frac{4R}{(1-r)^2}.
$$
By the Taylor formula we deduce that
$$
\left|F_p^{-1}(z)-F_0^{-1}(z)-p\left.\frac{d}{dp}F_p^{-1}(z)\right|_{p=0}\right|\le
C_rp^2,
$$
where
$C_r=\frac{8M^2R}{(1-r)^2}\left(1+\frac{16R}{(1-r)^3}\right)$. It
remains to note that due to Lemma~\ref{lem: sanity check} and the
identity~\eqref{eq:first der}, we have
$\left.\frac{d}{dp}F_p^{-1}(z)\right|_{p=0}=\left(z-H_\eta(\sin
z)\right)\cos z$.
\end{proof}

\begin{proof}[Proof of Theorem~\ref{thm:abs}] We will fix from now on
some $r\in (9/10,1)$.
Note that since the Hermite polynomial $h_5$ is odd, so is $H_\eta$.
Hence also $F_p$ is odd, and therefore $a_k(p)=0$ for even $k$. For
$z\in r\D$ write $\phi(z)=\left(z-H_\eta(\sin z)\right)\cos z$.
Consider the power series expansions
$$
\sin z=\sum_{k=0}^\infty
b_{2k+1} z^{2k+1}=\sum_{k=0}^\infty
\frac{(-1)^k}{(2k+1)!} z^{2k+1},
$$ and
\begin{equation}\label{eq:c series}
\phi(z)=\sum_{k=0}^\infty c_{2k+1}
z^{2k+1}.
\end{equation}

By the Cauchy formula we have for every $k\in
\N\cup\{0\}$,
\begin{equation}\label{eq:all coef}
\left|a_{2k+1}(p)-b_{2k+1}-pc_{2k+1}\right|= \left|\frac{1}{2\pi i
r}\oint_{r\partial \D}\frac{F_p^{-1}(z)-\sin z-
p\phi(z)}{z^{2k+2}}dz\right|\stackrel{\eqref{eq:perturbed
taylor}}{\le} \frac{C_rp^2}{r^{2k+2}}.
\end{equation}
Note that by Lemma~\ref{lem:analytic} the radius of convergence of
the series in~\eqref{eq:c series} is at least $1$, and therefore
$\sum_{k=0}^\infty |c_{2k+1}|(9/10)^{2k+1}<\infty$. Hence,
\begin{multline}\label{eq:gamma exists}
\sum_{k=0}^\infty
|a_{2k+1}(p)|\left(\frac{9}{10}\right)^{2k+1}\stackrel{\eqref{eq:all
coef}}{\ge} \sum_{k=0}^\infty
\frac{(9/10)^{2k+1}}{(2k+1)!}-p\sum_{k=0}^\infty
|c_{2k+1}|\left(\frac{9}{10}\right)^{2k+1}-\frac{C_rp^2}{r}\sum_{k=0}^\infty
\left(\frac{9}{10r}\right)^{2k+1}\\=\frac{e^{9/10}-e^{-9/10}}{2}-O(p)>1.02-O(p).
\end{multline}

By continuity, it follows from~\eqref{eq:gamma exists} that provided
$p$ is small enough there exists $\gamma>0$ satisfying the
identity~\eqref{eq:1}. Our goal is to prove~\eqref{eq:c}, so assume
for contradiction that $\gamma\le\log\left(1+\sqrt{2}\right)<9/10$.
Note that since  $r\in (9/10,1)$ we have
\begin{equation}\label{eq:49/50}
\frac{\gamma}{r}\le
\frac{10\log\left(1+\sqrt{2}\right)}{9}<\frac{49}{50}.
\end{equation}

Fix $\e>0$ that will be determined later. We have seen in Lemma~\ref{lem:1-1} that $\sin\left(\frac{9}{10}\D\right)\subseteq \S$. Since $H_\eta$ is analytic on $\S$, it follows that $\phi$ is analytic on $\frac{9}{10}\D$. Since $\gamma<9/10$, there exists $n\in \N$ satisfying
\begin{equation}\label{eq:n}
\sum_{k=n+1}^\infty |c_{2k+1}|\gamma^{2k+1}<\frac{\e}{2}.
\end{equation}
There exists
$p=p(\e)$ such that for all $p\in (0,p(\e))$ we have
$p|c_{2k+1}|<\frac12 |b_{2k+1}|$ for all $k\in \{0,\ldots,n\}$. In
particular, we have
$\sign(b_{2k+1}+pc_{2k+1})=\sign(b_{2k+1})=(-1)^k$. Now,
\begin{multline}\label{eq:past to first order}
\left|1-\frac{F_p^{-1}(i\gamma)}{i}\right|\stackrel{\eqref{eq:1}}{=}\left|\sum_{k=0}^\infty
\left(|a_{2k+1}(p)|-(-1)^ka_{2k+1}(p)\right)\gamma^{2k+1}\right|\\\stackrel{\eqref{eq:all coef}}{\le}
\sum_{k=0}^\infty
\left||b_{2k+1}+pc_{2k+1}|-(-1)^k\left(b_{2k+1}+pc_{2k+1}\right)\right|\gamma^{2k+1}+2\sum_{k=0}^\infty
\frac{C_rp^2}{r^{2k+2}}\gamma^{2k+1}.
\end{multline}
To estimate the two terms on the right hand side on~\eqref{eq:past to first order}, note first that
\begin{equation}\label{eq:Cr'}
2\sum_{k=0}^\infty
\frac{C_rp^2}{r^{2k+2}}\gamma^{2k+1}\stackrel{\eqref{eq:49/50}}{\le} \frac{2C_r}{r}p^2\sum_{k=0}^\infty\left(\frac{49}{50}\right)^{2k+1}\le C_r'p^2,
\end{equation}
where $C_r'$ depends only on $r$.
Since $p\in (0,p(\e))$ we know that for all $k\in \{0,\ldots,n\}$ we have  $|b_{2k+1}+pc_{2k+1}|=(-1)^k\left(b_{2k+1}+pc_{2k+1}\right)$. Hence the first $n$ terms of the first sum in the right hand side of~\eqref{eq:past to first order} vanish. Therefore,
\begin{multline}\label{eq:use n}
\sum_{k=0}^\infty
\left||b_{2k+1}+pc_{2k+1}|-(-1)^k\left(b_{2k+1}+pc_{2k+1}\right)\right|\gamma^{2k+1}\\=
\sum_{k=n+1}^\infty \left||b_{2k+1}+pc_{2k+1}|-|b_{2k+1}|-(-1)^kpc_{2k+1}\right|\gamma^{2k+1}\le 2p\sum_{k=n+1}^\infty |c_{2k+1}|\gamma^{2k+1}\stackrel{\eqref{eq:n}}{<}p\e.
\end{multline}
By substituting~\eqref{eq:Cr'} and~\eqref{eq:use n} into~\eqref{eq:past to first order}, we see that if we define $\beta=F^{-1}_p(i\gamma)-i$ then
\begin{equation}\label{eq:beta}
|\beta|\le C_r'p^2+p\e.
\end{equation}
Let $L_0$ be the Lipschitz constant of $H_0$ on $i+\frac12
\D\subseteq \S$ (the disc of radius $\frac12$ centered at $i$).
Similarly let $L_\eta$ be the Lipschitz constant of $H_\eta$ on
$i+\frac12 \D$, and set $L=\max\{L_0,L_\eta\}$. It follows that
$F_p=(1-p)H_0+pH_\eta$ is $L$-Lipschitz on $i+\frac12 \D$. Due
to~\eqref{eq:beta}, if $p$ is small enough then $i+\beta\in
i+\frac12 \D$, and therefore,
\begin{multline*}
\log\left(1+\sqrt{2}\right)\ge\gamma=\frac{F_p(\beta+i)}{i}\ge
\frac{F_p(i)}{i}-L|\beta|\stackrel{\eqref{eq:beta}}{\ge}\frac{(1-p)H_0(i)+pH_\eta(i)}{i}-Lp\left(C_r'p+\e\right)\\=
(1-p)\log\left(1+\sqrt{2}\right)+p\frac{H_\eta(i)}{i}-Lp\left(C_r'p+\e\right).
\end{multline*}
This simplifies to give the following estimate:
$$
\frac{H_\eta(i)}{i}\le \log\left(1+\sqrt{2}\right)+ LC_r'p+L\e.
$$
Since this is supposed to hold for all $\e>0$ and $p\in (0,p(\e))$, we arrive at a contradiction to
Theorem~\ref{thm:konig}.
\end{proof}

\begin{remark}\label{rem:inspection}
{\em An inspection of the proof of Theorem~\ref{thm:win!} shows that
the only property of $H_\eta$ that was used is that it is a
counterexample to K\"onig's problem. In other words, assume that
$f,g:\R^2\to\{-1,1\}$ are measurable functions and consider the
function $H:\S\to\C$ given by
\begin{equation*}
H(z)=\frac{1}{2\pi(1-z^2)}\int_{\R^2\times
\R^2}f(x)g(y)\exp\left(\frac{-\|x\|_2^2-\|y\|_2^2+2z\langle
x,y\rangle}{1-z^2}\right)dxdy.
\end{equation*}
Assume that $B_K(f,g)>4\pi \log\left(1+\sqrt{2}\right)$, where $B_K$
is K\"onig's bilinear form given in~\eqref{eq:B_k}. Then one can
repeat the proof of Theorem~\ref{thm:win!} with $H_\eta$ replaced by
$H$, arriving at the same conclusion.}
\end{remark}

\begin{conjecture}\label{conj:no p}
Recalling Definition~\ref{def:kri round} and Corollary~\ref{cor:krivine}, we conjecture that for small enough $\eta\in (0,1)$, the pair of functions $f=g=f_\eta:\R^2\to \{-1,1\}$ is a Krivine rounding scheme for which we have $c(f_\eta,f_\eta)>\frac{2}{\pi}\log\left(1+\sqrt{2}\right)$. In other words, we conjecture that in order to prove Theorem~\ref{thm:win!} we do not need to use a convex combination of $H_0$ and $H_\eta$ as we did above, but rather use only $H_\eta$ itself.
\end{conjecture}

\section{Proof of Theorem~\ref{thm:KTJ}}\label{sec:KTJ}

For  the final equality
in~\eqref{eq:KT}, see equation (2.4) in~\cite{Kon00}. Denote
\begin{multline}\label{eq:def M}
M\eqdef \sup_{\substack{f,g\in L_\infty(0,\infty)\\ f,g\in [-1,1]\
\mathrm{a.e.}}} \int_{0}^\infty\int_0^\infty
f(x)g(y)\exp\left(-\frac{x^2+y^2}{2}\right)\sin(xy)dx
dy\\=\sup_{\substack{f,g\in L_\infty(0,\infty)\\ f,g\in \{-1,1\}\
\mathrm{a.e.}}} \int_{0}^\infty\int_0^\infty
f(x)g(y)\exp\left(-\frac{x^2+y^2}{2}\right)\sin(xy)dx dy.
\end{multline}
Theorem~\ref{thm:KTJ} will follow once we show that $M$ is attained
when $f=g=1$ or $f=g=-1$. Note that the supremum in~\eqref{eq:def M}
is attained at some $f,g:(0,\infty)\to \{-1,1\}$. Indeed, let $\mu$
be the measure on $(0,\infty)$ with density $x\mapsto e^{-x^2/2}$.
If $f_n,g_n:(0,\infty)\to [-1,1]$ satisfy $ \lim_{n\to\infty}
\int_{0}^\infty\int_0^\infty
f_n(x)g_n(y)\exp\left(-\frac{x^2+y^2}{2}\right)\sin(xy)dx dy=M$ then
by passing to a subsequence we may assume that there are $f,g\in
L_2(\mu)$ such that $f_n$ converges weakly to $f$ and $g_n$
converges weakly to $g$ in $L_2(\mu)$. Then $f,g\in [-1,1]$ a.e. and
by weak convergence we have $\int_{0}^\infty\int_0^\infty
f(x)g(y)\exp\left(-\frac{x^2+y^2}{2}\right)\sin(xy)dx dy=M$. A
simple extreme point argument shows that we may also assume that
$f,g\in \{-1,1\}$.

Assume from now on that $f,g:(0,\infty)\to \{-1,1\}$ are maximizers
of~\eqref{eq:def M}, i.e.,
\begin{equation}\label{eq:maximizer}
M= \int_{0}^\infty\int_0^\infty
f(x)g(y)\exp\left(-\frac{x^2+y^2}{2}\right)\sin(xy)dxdy.
\end{equation}
This implies that for almost every $x\in (0,\infty)$ we have
\begin{equation}\label{eq:f=sign}
f(x)=\sign\left(\int_0^\infty g(y)e^{-y^2/2}\sin(xy)dy\right),
\end{equation}
and
\begin{equation}\label{eq:g=sign}
 g(x)=\sign\left(\int_0^\infty
f(y)e^{-y^2/2}\sin(xy)dy\right).
\end{equation}
Consequently,
\begin{multline}\label{eq:absolute fg}
M=\int_0^\infty \left|\int_0^\infty
g(y)e^{-y^2/2}\sin(xy)dy\right|e^{-x^2/2}dx\\=\int_0^\infty
\left|\int_0^\infty f(y)e^{-y^2/2}\sin(xy)dy\right|e^{-x^2/2}dx.
\end{multline}

\begin{lemma}\label{lem:LipI}
Define $I:[0,\infty)\to \R$ by
\begin{equation}\label{eq:def I}
I(y)\eqdef \int_0^\infty
f(x)\exp\left(-\frac{x^2+y^2}{2}\right)\sin(xy)dx.
\end{equation}
Then $I$ is $\sqrt{2}$-Lipschitz.
\end{lemma}

\begin{proof} The Lipschitz condition follows from the following
simple estimate on $|I'|$:
\begin{eqnarray}\label{eq:root of 2}
\nonumber |I'(y)| &\le& \int_0^\infty
\exp\left(-\frac{x^2+y^2}{2}\right)\left|x\cos(xy)-y\sin(xy)\right|dx
\\&\le& \nonumber \int_0^\infty
\exp\left(-\frac{x^2+y^2}{2}\right)\sqrt{x^2+y^2}dx\\
&\le&\nonumber
\sqrt{2}e^{-y^2/2}\int_0^\infty e^{-x^2/2}\max\{x,y\}dx\\
\nonumber&=&\sqrt{2}e^{-y^2/2}\left(\int_0^ye^{-x^2/2}ydx+\int_y^\infty
e^{-x^2/2}xdx\right)\\&\le&
\sqrt{2}e^{-y^2/2}\left(y^2+e^{-y^2/2}\right)\le \sqrt{2},
\end{eqnarray}
where in the last inequality in~\eqref{eq:root of 2} we used the
bound $e^{-t}(2t+e^{-t})\le 1$, which holds for all $t\ge 0$ since
$e^{t}-e^{-t}=\sum_{k=0}^\infty \frac{2t^{2k+1}}{(2k+1)!}\ge 2t$.
\end{proof}

\begin{lemma}\label{lem:y0}
For every $z\in [2/5,4/3]$ we have
\begin{equation}\label{eq:integratedI}
\int_{z-\frac14}^{z+\frac14}
\int_0^\infty\exp\left(-\frac{x^2+y^2}{2}\right)\sin(xy)dx dy>
\frac{3}{20}.
\end{equation}
\end{lemma}
\begin{proof}
Since for every $q\ge 0$ we have $\sin (q)\ge
q-\frac{q^3}{6}+\frac{q^5}{120}-\frac{q^7}{5040}$,
\begin{eqnarray}\label{eq:use sin taylor}
&&\!\!\!\!\!\!\!\!\!\!\!\!\!\!\!\!\!\!\!\!\!\!\!\!\nonumber\int_0^\infty\exp\left(-\frac{x^2+y^2}{2}\right)\sin(xy)dx\\ \nonumber&\ge&
\int_0^\infty\exp\left(-\frac{x^2+y^2}{2}\right)\left(xy-\frac{(xy)^3}{6}+\frac{(xy)^5}{120}-\frac{(xy)^7}{5040}\right)dx\\&=&
-\frac{ye^{-y^2/2}}{105}\left(y^6-7y^4+35y^2-105\right)\eqdef h(y).
\end{eqnarray}
We claim that there is a unique $w\in [0,2]$ at which $h'$ vanishes,
and moreover $h$ attains its maximum on $[0,2]$ at $w$. Indeed,
$h'(0)=e^{-y^2/2}\left(y^8-14y^6+70y^4-210y^2+105\right)/105$, and
therefore $h'(y)=1>0$ and $h'(2)=-17/(7e^2)<0$, so it suffices to
show that $h'$ can have at most one zero on $[0,2]$. To this end it
suffices to show that the polynomial
$p(y)=y^8-14y^6+70y^4-210y^2+105$ is monotone on $[0,2]$. This is
indeed the case since for $y\in [0,2]$ we have $ p'(y)=
-8y^5(4-y^2)-y(52\left(y^2-35/13\right)^2+560/13)<0$.

The above reasoning shows that if we set
$G(z)=\int_{z-1/4}^{z+1/4}h(y)dy$ then $G$ does not have local
minima on $[2/5,4/3]$. Indeed, if $z\in [2/5,4/3]$ satisfies
$G'(z)=h(z+1/4)-h(z-1/4)=0$ then $z-1/4<w<z+1/4$, implying that
$h'(z-1/4)>0$ and $h'(z+1/4)<0$. Hence $G''(z)=h'(z+1/4)-h'(z-1/4)<0$,
so $z$ cannot be a local minimum of $G$. Now, for every $z\in
[2/5,4/3]$ we have
$$
\int_{z-\frac14}^{z+\frac14}
\int_0^\infty\exp\left(-\frac{x^2+y^2}{2}\right)\sin(xy)dx
dy\stackrel{\eqref{eq:use sin taylor}}{\ge} G(z)\ge
\min\left\{G\left(\frac25\right),G\left(\frac43\right)\right\}>
\frac{3}{20},
$$
where we used the fact that the above values of $G$ can be computed
in closed form, for example
$G(4/3)=(19047383/313528320)e^{-169/288}+(131938921/313528320)e^{-361/288}>0.153$.
\end{proof}

We will consider the following two quantities:
\begin{equation}\label{eq:def M0}
M_0\eqdef\int_0^\infty\int_0^\infty \exp\left(-\frac{x^2+y^2}{2}\right)\sin(xy)dx
dy= \frac{\log\left(1+\sqrt{2}\right)}{\sqrt{2}}=0.6232...,
\end{equation}
and
\begin{equation}\label{eq:def M1}
M_1\eqdef\int_0^\infty\int_0^\infty \exp\left(-\frac{x^2+y^2}{2}\right)|\sin(xy)|dx
dy.
\end{equation}
Our goal is to show that $M=M_0$. Clearly $M_0\le M_1$. Our next
(technical) lemma shows that $M_1$ is actually quite close to $M_0$.

\begin{lemma}\label{lem:M0M1}
$M_1-M_0<\frac{1}{20}$.
\end{lemma}

\begin{proof}
Since the integral in~\eqref{eq:def M1} converges very quickly and
can therefore be computed numerically with high precision,
Lemma~\ref{lem:M0M1} does not have much content. Nevertheless, we
wish to explain how to reduce this lemma to an evaluation of an
integral that can be computed in closed form. Let $p_n$ be the
Taylor polynomial of degree $2n-1$ of the function $x\mapsto
\sqrt{1-x}$, i.e.,
$$
p_n(x)=\sum_{k=0}^{2n-1}(-1)^k\binom{1/2}{k}x^k=1+\sum_{k=1}^{2n-1}\frac{\prod_{j=0}^{k-1}(2j-1)}{2^kk!}x^k.
$$
Then $\sqrt{1-x}\le p_n(x)$ for all $x\in (-1,1)$ and $n\in \N$.

Now,
\begin{multline}\label{eq:cos taylor}
M_1=\int_0^\infty\int_0^\infty
\exp\left(-\frac{x^2+y^2}{2}\right)\sqrt{\frac{1-\cos(2xy)}{2}}dx
dy\\\le \frac{1}{\sqrt{2}}\int_0^\infty\int_0^\infty
\exp\left(-\frac{x^2+y^2}{2}\right)p_n(\cos(2xy))dx dy.
\end{multline}
For $j\in \N$ the integral $\int_0^\infty\int_0^\infty
\exp\left(-\frac{x^2+y^2}{2}\right)(\cos(2xy))^jdx dy$ can be
computed in closed form (it equals $\pi$ times a linear combination
with rational coefficients of square roots of integers). One can
therefore explicitly evaluate the integral in right hand side
of~\eqref{eq:cos taylor} for $n=11$, obtaining the bound
$M_1<0.671<M_0+0.05.$
\end{proof}

\begin{lemma}\label{lem:small const int}
$f$ and $g$ are constant on the interval $[2/5,4/3]$.
\end{lemma}

\begin{proof}
Assume for contradiction that $g$ is not constant on $[2/5,4/3]$.
By~\eqref{eq:g=sign} we know that $g=\sign(I)$, where $I$ is given
in~\eqref{eq:def I}. Hence there is some $z\in [2/5,4/3]$ such that
$I(z)=0$. By Lemma~\ref{lem:LipI} we therefore know that $|I(y)|\le
\sqrt{2}|y-z|$ for all $y\in [0,\infty)$. Hence,
\begin{multline*}
M_0\le M\stackrel{\eqref{eq:absolute fg}}{\le}
\int_{[0,z-1/4]\cup[z+1/4,\infty)}
\int_0^\infty\exp\left(-\frac{x^2+y^2}{2}\right)|\sin(xy)|dxdy
+\int_{z-\frac14}^{z+\frac14}|I(y)|dy\\
\le M_1-\int_{z-\frac14}^{z+\frac14}
\int_0^\infty\exp\left(-\frac{x^2+y^2}{2}\right)|\sin(xy)|dx
dy+\frac{\sqrt{2}}{16}\stackrel{\eqref{eq:integratedI}}{<}M_1-\frac{3}{20}+\frac{\sqrt{2}}{16}<M_1-\frac{3}{50}.
\end{multline*}
This contradicts Lemma~\ref{lem:M0M1}.
\end{proof}

Before proceeding to the proof of Theorem~\ref{thm:KTJ}, we record
one more elementary lemma.

\begin{lemma}\label{lem:a^3}
For every $a>0$ we have
\begin{equation}\label{eq:a^3}
\int_a^\infty e^{-x^2/2}dx<\frac{16}{3e^2a^3}.
\end{equation}
\end{lemma}

\begin{proof}
Set $c=16/(3e^2)$ and $\psi(a)=\frac{c}{a^3}-\int_a^\infty
e^{-x^2/2}dx$. Since $\lim_{a\to\infty}\psi(a)=0$, it suffices to show that $\psi$ is decreasing on
$(0,\infty)$. Because $\psi'(a)=e^{-a^2/2}-\frac{3c}{a^4}$, our goal is
to show that $e^{a^2/2}\ge a^4/(3c)$, or equivalently that $a^2\ge
8\log a-2\log(3c)$. Set $\rho(a) =a^2-8\log a+2\log(3c)$. Since
$\rho'(a)=2a-8/a$, the minimum of $\rho$ is attained at $a=2$. We
are done, since by the choice of $c$ we have $\rho(2)=0$.
\end{proof}

\begin{proof}[Proof of Theorem~\ref{thm:konig}]
Due to Lemma~\ref{lem:small const int}, we may assume from now on
that $f(x)=1$ for all $x\in [2/5,4/3]$. Our goal is to show that
under this assumption $f=g=1$. We shall achieve this in a few steps.

We will first show that $f(y)=g(y)=1$ for all $y\in [0,4/3]$. To see
this, fix $y\in (0,1/2]$ and note that the function $x\mapsto
\sin(xy)$ is concave on $[0,4/3]$ (since in this range we have
$xy\le 2/3\le \pi/2$). It follows that for $x\in [0,4/3]$ we have
$\sin(xy)\ge \frac{\sin(4y/3)}{4/3}x$ and for $x\in [4/3,\infty)$ we
have $|\sin(xy)|\le \frac{\sin(4y/3)}{4/3}x$. Using this fact, the
fact that $|\sin(xy)|\le xy$ for all $x\ge 0$, and the fact that
$f=1$ on $[2/5,4/3]$, we deduce that
\begin{eqnarray}\label{eq:r}
\nonumber&&\!\!\!\!\!\!\!\!\!\!\!\!\!\!\!\!\!\!\!\!\!\!\!\!\!\!\!\int_0^\infty e^{-x^2/2}\sin(xy)f(x)dx\\&\ge&\nonumber
-\int_0^{\frac25}e^{-x^2/2}xydx+\int_{\frac25}^{\frac43}
e^{-x^2/2}\frac{\sin(4y/3)}{4/3}xdx-\int_{\frac43}^\infty
e^{-x^2/2}\frac{\sin(4y/3)}{4/3}xdx\\
&=&
-\left(1-e^{-2/25}\right)y+\frac{\sin(4y/3)}{4/3}\left(e^{-2/25}-2e^{-8/9}\right)\eqdef
r(y).
\end{eqnarray}
Note that for $z\in [0,1/2]$ we have
\begin{multline*}
r'(z)=\cos\left(\frac{4z}{3}\right)\left(e^{-2/25}-2e^{-8/9}\right)-\left(1-e^{-2/25}\right)\\\ge
\cos\left(\frac23\right)\left(e^{-2/25}-2e^{-8/9}\right)-\left(1-e^{-2/25}\right)>0.002>0.
\end{multline*}
Hence $r$ in increasing on $[0,1/2]$, and in particular
$r(y)>r(0)=0$. By~\eqref{eq:g=sign} and~\eqref{eq:r}, it follows
that $g(y)=1$ for all $y\in [0,1/2]$. Since Lemma~\ref{lem:small
const int} tells us that $g$ is constant on $[2/5,4/3]$, it follows
that $g=1$ on $[0,4/3]$. We my now repeat the above argument with
the roles of $f$ and $g$ interchanged, deducing that $f=g=1$ on
$[0,4/3]$.

Our next goal is to show that $f=g=1$ on $[0,3\pi/4]$. We already
know that $f=g=1$ on $[0,4/3]$, so assume that $y\in [4/3,3\pi/4]$.
Then $4y/3\ge (4/3)^2>\pi/2$ and $4y/3\le \pi$, implying that
\begin{multline}\label{eq:8/27}
\int_0^{\frac43}f(x)e^{-x^2/2}\sin(xy)dx\ge
\int_0^{\frac{\pi}{2y}}e^{-x^2/2}\sin(xy)dx\stackrel{(*)}{\ge}
\int_0^{\frac{\pi}{2y}}e^{-x^2/2}\frac{2xy}{\pi}dx\\=\frac{2y}{\pi}\left(1-e^{-\pi^2/(8y^2)}\right)
\stackrel{(**)}{\ge}
\frac{2y}{\pi}\left(\frac{\pi^2}{8y^2}-\frac{\pi^4}{128y^4}\right)=\frac{\pi}{4y}-\frac{\pi^3}{64y^3}\stackrel{(***)}{\ge}
\frac{8}{27},
\end{multline}
where in $(*)$ we used the fact that $\sin(z)\ge 2z/\pi$ for $z\in
[0,\pi/2]$, in $(**)$ we used the elementary inequality $1-e^{-z}\ge
z-z^2/2$, which holds for all $z\ge 0$, and in $(***)$ we used the
fact that the function $y\mapsto \frac{\pi}{4y}-\frac{\pi^3}{64y^3}$
has a unique local maximum on $(1,\infty)$, which implies that its
minimum on $[4/3,3\pi/4]$ is attained at the endpoints, and
therefore its minimum on $[4/3,3\pi/4]$ equals $8/27$.

Now,
\begin{multline}\label{eq:g pos a bit bigger}
\int_0^\infty
f(x)e^{-x^2/2}\sin(xy)dx\stackrel{\eqref{eq:8/27}}{\ge}
\frac{8}{27}-\int_{\frac43}^\infty
e^{-x^2/2}dx=\frac{8}{27}-\left(\sqrt{\frac{\pi}{2}}-\int_0^{\frac43}e^{-x^2/2}dx\right)\\\ge
\frac{8}{27}-\sqrt{\frac{\pi}{2}}+\int_0^{\frac43}\left(1-\frac{x^2}{2}+\frac{x^4}{8}
-\frac{x^6}{48}\right)dx=\frac{302572}{229635}-\sqrt{\frac{\pi}{2}}>0.
\end{multline}
Using~\eqref{eq:g=sign} we deduce from~\eqref{eq:g pos a bit bigger}
that $g(y)=1$. By symmetry the same argument applies to $f$,
implying that $f=g=1$ on $[0,3\pi/4]$.

Let $A\ge 3\pi/4$ be the supremum over those $a>0$ such that $f=g=1$
on $[0,a]$. Our goal is to show that $A=\infty$, so assume for
contradiction that $A$ is finite.

Note that for every $y>0$ and every $k\in \N\cup\{0\}$,
\begin{equation}\label{eq:k}
\int_{\frac{2k\pi}{y}}^{\frac{2(k+1)\pi}{y}}e^{-x^2/2}\sin(xy)dx=
\int_{\frac{2k\pi}{y}}^{\frac{(2k+1)\pi}{y}}\left(e^{-x^2/2}-e^{-(x+\pi/y)^2/2}\right)\sin(xy)dx\ge
0.
\end{equation}
It follows that if $A\in [2k\pi/y,2(k+1)\pi/y]$ then
\begin{equation}\label{eq:A between}
\int_{\frac{2k\pi}{y}}^{A}e^{-x^2/2}\sin(xy)dx\ge 0.
\end{equation}
To check~\eqref{eq:A between}, note that for $A\in
[2k\pi/y,(2k+1)\pi/y]$ the integrand in~\eqref{eq:A between} is
nonnegative, and for $A\in [(2k+1)\pi/y,2(k+1)\pi/y]$ the integral
in~\eqref{eq:A between} is at least the integral in the left hand
side of~\eqref{eq:k}. Assume from now on that $y>A$ and note that
since $A\ge 3\pi/4$ we have $3\pi/(2y)<A$. This implies the
following bound:
\begin{eqnarray}\label{eq:first interval}
\nonumber\int_0^{\min\left\{\frac{2\pi}{y},A\right\}}e^{-x^2/2}\sin(xy)dx&\ge&
\int_0^{\frac{2\pi}{y}}e^{-x^2/2}\sin(xy)dx\\\nonumber
&=&\int_{0}^{\frac{\pi}{y}}\left(e^{-x^2/2}-e^{-(x+\pi/y)^2/2}\right)\sin(xy)dx\\
&=&\nonumber \int_{0}^{\frac{\pi}{y}}e^{-x^2/2}\left(1-e^{-\pi
x/y}e^{-(\pi/y)^2/2}\right)\sin(xy)dx\\
&\ge &\nonumber
\left(1-e^{-(\pi/y)^2/2}\right)\int_0^{\frac{\pi}{2y}} e^{-x^2/2}
\frac{2xy}{\pi}dx\\
&=&\nonumber
\frac{2y}{\pi}\left(1-e^{-\pi^2/(2y^2)}\right)\left(1-e^{-\pi^2/(8y^2)}\right)\\
&\ge&
\frac{2y}{\pi}\left(\frac{\pi^2}{2y^2}-\frac{\pi^4}{8y^4}\right)
\left(\frac{\pi^2}{8y^2}-\frac{\pi^4}{128y^4}\right)\ge
\frac{5\pi^3}{81y^3},
\end{eqnarray}
where in the last inequality of~\eqref{eq:first interval} we used
the fact that $y>A\ge 3\pi/4$.

A combination of~\eqref{eq:k}, \eqref{eq:A between} and
~\eqref{eq:first interval} with the fact that $f=1$ on $[0,A)$ shows
that
$$
\int_0^Af(x)e^{-x^2/2}\sin(xy)dx\ge \frac{5\pi^3}{81y^3},
$$
and therefore, by Lemma~\ref{lem:a^3},
\begin{equation}\label{eq:cubed pos}
\int_0^\infty f(x)e^{-x^2/2}\sin(xy)dx\ge
\frac{5\pi^3}{81y^3}-\int_A^\infty
e^{-x^2/2}dx\stackrel{\eqref{eq:a^3}}{\ge}
\frac{5\pi^3}{81y^3}-\frac{16}{3e^2A^3}. \end{equation} The right
hand side of~\eqref{eq:cubed pos} is positive provided $y\le 5A/4$.
By~\eqref{eq:g=sign} this means that $g=1$ on $[A,5A/4]$, and hence
also on $[0,5A/4]$. By symmetry, $f=1$ on  $[0,5A/4]$ as well,
contradicting the definition of $A$.
\end{proof}

\bigskip


\bibliographystyle{abbrv}

\bibliography{noKrivine}
\end{document}
Let $p_n$ be the Taylor polynomial of degree $2n-1$ for the function $x\mapsto \sqrt{1-x}$, i.e.,
$$
p_n(x)=\sum_{k=0}^{2n-1}(-1)^k\binom{1/2}{k}x^k=1+\sum_{k=1}^{2n-1}\frac{\prod_{j=0}^{k-1}(2j-1)}{2^kk!}x^k.
$$
Then $\sqrt{1-x}\le p_n(x)$ for all $x\in (-1,1)$ and $n\in \N$. Moreover, $p_n$ is decreasing on $(-1,1)$, since $p_n'$ is a polynomial with negative coefficients, which is therefore clearly decreasing on $(0,1)$. Moreover, $p_n'(-x)$ has alternating decreasing coefficients with constant term $-\frac12$, which shows that $p_n'$ is negative on $(-1,0)$ as well.

Now, let $q_n$ be the Taylor polynomial of $\cos(x)$ of degree $4n-2$, i.e.,
$$
q_n(x)=\sum_{k=0}^{4n-2}\frac{(-1)^k}{(2k)!}x^{2k}.
$$

Then for $x\in (-1,1)$,
$$
p_n(x)=\sqrt{1-x}+\frac{\prod_{j=1}^{2n}(2j-1)}{2^{2n}(2n-1)!}\int_0^x(x-t)^{2n-1}\frac{1}{(1-t)^{\frac12-2n}}dt
$$


\begin{lemma}\label{lem:M0M1}
$M_1-M_0<e^{-\pi}/2<0.05$.
\end{lemma}

\begin{proof}
Note that if $(x,y)\in (0,\infty)^2$ and $x^2+y^2\le \pi$ then $xy\in [0,\pi]$, whence  $\sin(xy)\ge 0$. Now,
\begin{multline*}
M_1-M_0\stackrel{\eqref{eq:def M0}\wedge\eqref{eq:def M1}}{=} 2\int_{(0,\infty)^2\setminus \left(\sqrt{\pi}\D\right)}\exp\left(-\frac{x^2+y^2}{2}\right)\max\{\sin(xy),0\}dxdy\\\le
2\int_{(0,\infty)^2\setminus \left(\sqrt{\pi}\D\right)}\exp\left(-\frac{x^2+y^2}{2}\right)dxdy=2\int_{\sqrt{\pi}}^\infty\int_0^{\pi/4}re^{-r^2/2}d\theta dr =\frac{\pi e^{-\pi/2}}{4}.
\end{multline*}
\end{proof}